\documentclass[12pt,oneside]{jdg-p}
\usepackage{latexsym}
\usepackage{amssymb, latexsym}
\usepackage{amsmath}
\usepackage{mathrsfs}
\usepackage{amsxtra}
\usepackage{amsthm}
\usepackage{setspace}
\usepackage{epsfig}
\usepackage{geometry}
\usepackage{graphicx}

\newtheorem{theorem}{Theorem}[section]
\newtheorem{corollary}[theorem]{Corollary}

\newtheorem{lemma}[theorem]{Lemma}

\newtheorem{proposition}[theorem]{Proposition}

\theoremstyle{definition}
\newtheorem{definition}[theorem]{Definition}

\newtheorem{remark}[theorem]{Remark}

\title[Lower Bounds for the first eigenvalue of the Laplacian]{Lower Bounds for the first eigenvalue of the Laplacian with Dirichlet Boundary Conditions in a Hyperbolic Space of a Negative Constant Curvature}

\author{Sergei Artamoshin}

\thanks{The author wants to thank professor J\a'ozef Dodziuk for being a scientific advisor, for his interest to this paper and for his patience.}

\begin{document}

\begin{abstract}
 In this paper we consider a domain in a space of negative constant sectional curvature. Such assumption about the sectional curvature let us develop a new technique and improve existing lower bounds of eigenvalues from Dirichlet eigenvalue problem, obtained by Alessandro Savo in 2009.
\end{abstract}

\maketitle

\begin{spacing}{1.5}

\section{Introduction}

In this paper we compute lower bounds for the smallest positive eigenvalue of a Dirichlet Eigenvalue Problem in a domain of a constant negative curvature. Such estimates for riemannian manifolds have been discussed in many papers. For the bibliography, see, for example, \cite{Chavel} or \cite{Cheng}. Relatively recent result was obtained by Alessandro Savo in \cite{Savo}. Here we are going to introduce a new method to analyze eigenvalues and improve the existing estimations in a space of a constant sectional curvature. The Rayleigh's Theorem in \cite{Chavel}, p.~16, let us reduce the discussion to the estimation of lower bounds for the smallest positive eigenvalue in the smallest circumscribed disc. Recall that the Dirichelt Eigenvalue Problem for a disc of radius $\delta$ is formulated as follows. Find all $\lambda\in\mathbb{R}$ and corresponding functions $\varphi_\lambda(\upsilon, r)$, which are eigenfunctions of Hyperbolic Laplacian in $(k+1)-$dimensional hyperbolic space with a constant sectional curvature $\kappa=-1/\rho^2$ such that
\begin{equation}\label{Dirichlet_Original}
\left\{
  \begin{array}{ll}
     & \hbox{$\triangle\varphi_\lambda(\upsilon,r)+\lambda \varphi_\lambda(\upsilon,r)=0 \quad \forall r\in [0,\delta], \,\,\lambda - \text{real}$;} \\
     & \hbox{$\varphi_\lambda(\upsilon, \delta)=0$,}
  \end{array}
\right.
\end{equation}
where $\upsilon$ is a point on the unit sphere $\sigma_O^k$ centered at the origin. Below are the basic estimates and results obtained for a hyperbolic disc in this paper.

\begin{enumerate}
  \item    H.P. McKean showed that
\begin{equation}\label{Mckean_Estimation}
		\lambda\geq -\frac{\kappa k^2}{4}\quad\text{for all}\quad \delta>0\,,
\end{equation}
see \cite{McKean} or \cite{Chavel} (p.46). From \cite{Chavel} (p.46) we also know B.Randol's result stating that
\begin{equation}
\lim\limits_{\delta\rightarrow +\infty}\lambda(\delta)=-\frac{\kappa k^2}{4}\,.
\end{equation}
In 2009 Alessandro Savo obtained the following estimates for $\kappa=-1$.
\begin{equation}\label{Savo_Inequality}
    \frac{k^2}{4}+\frac{\pi^2}{\delta^2}-\frac{4\pi^2}{k\delta^3}\leq\lambda\leq  \frac{k^2}{4}+\frac{\pi^2}{\delta^2}+\frac{C}{\delta^3}\,,
\end{equation}
where $C=\frac{\pi^2(k^2+2k)}{2}\int\limits_0^\infty\frac{r^2}{\sinh^2r}dr$ and for $k=2$,
\begin{equation}\label{savo_equality}
    \lambda(\delta)=-\kappa+\frac{\pi^2}{\delta^2}\,,
\end{equation}
see \cite{Savo}, p.60, Theorem 5.6.
In this paper we shall see that the smallest eigenvalue $\lambda$ in problem \eqref{Dirichlet_Original} for $k=1$ must satisfy the following inequalities
\begin{equation}\label{Upper_Bound_Intro}
    -\kappa\frac{k^2}{4}+\left(\frac{\pi}{2\delta}\right)^2<\lambda_{\min}<
    -\kappa\frac{k^2}{4}+\left(\frac{\pi}{\delta}\right)^2 \,.
\end{equation}
For $k\geq3$ the lower bound in \eqref{Upper_Bound_Intro} can be improved to
\begin{equation}\label{Intro-4-3}
    \frac{-\kappa k^2}{4}+\left(\frac{\pi}{\delta}\right)^2<\lambda_{\min}\,
\end{equation}
and for $k=2$, the new technique presented in this paper also yields \eqref{savo_equality}. Note that \eqref{Intro-4-3} is the improvement of Savo's lower bound in \eqref{Savo_Inequality} for $k>2$ and for all $\delta$. The lower bound in  \eqref{Upper_Bound_Intro} yields a better estimate than the lower bound in \eqref{Savo_Inequality} for $k=1$ and for small $\delta$.

Note also that for $k=1$, the upper bound in \eqref{Upper_Bound_Intro} is stronger than the upper bound in \eqref{Savo_Inequality} as well as it is stronger than the upper bound obtained by S.Y. Cheng, see \cite{Cheng} or \cite{Chavel} (p.82), but still, remains weaker then the upper bound obtained by M. Gage, see \cite{Gage} or \cite{Chavel} (p. 80).

  \item We shall see in Theorem~\ref{Lower-and-Upper-bound-Theorem}, p.~\pageref{Lower-and-Upper-bound-Theorem} that in the hyperbolic space of three dimensions all the radial Dirichlet eigenfunctions together with their eigenvalues can be computed explicitly, i.e., the set of the following formulae

\begin{equation}\label{Intro-6}
    \lambda_j=-\kappa+\left(\frac{\pi j}{\delta}\right)^2\,,\,\,j=1,2,...
\end{equation}
yields the whole spectrum for the Dirichlet eigenvalue problem \eqref{Dirichlet_Original} restricted by a non-zero condition at the origin. The radial eigenfunction assuming the value 1 at the origin for each $\lambda_j$ can be written as

\begin{equation}\label{Intro-7}
    \varphi_{\lambda_j}(r)=\frac{\delta(\rho^2-\eta^2)}{2\pi\rho^2 \eta j}\cdot\sin\left(\frac{\pi jr}{\delta}\right)=
    \frac{\delta\sin(\pi jr/\delta)}{\pi j\rho\sinh(r/\rho)}\,,
\end{equation}
where $0\leq r\leq\delta<\infty$ and $\eta=\rho\tanh(r/2\rho)$. The last too equations \eqref{Intro-6} and \eqref{Intro-7} can be obtained by solving a proper ODE explicitly, but we are going to use a different technique developed in the paper.

\end{enumerate}

\section{Statement of results}

In this section we state the lower and the upper bounds for the minimal positive eigenvalue of a Dirichlet Eigenvalue Problem stated in a connected compact domain $\overline{M^n}\subseteq H^n$ with $\partial M^n\neq\emptyset$. According to~\cite{Chavel}, p.~8, the Dirichlet Eigenvalue Problem for $\overline{M^n}$ is stated as follows.

\textbf{\underline{Dirichlet Eigenvalue Problem:}}\label{Dirichlet-Eigen-General-Chavel} Let $M^n$ be relatively compact and connected domain with smooth boundary $\partial M^n\neq\emptyset$ and $\lambda_{\min}$ denotes the minimal eigenvalue. We are looking for all real numbers $\lambda$ for which there exists a nontrivial solution $\varphi\in C^2(M^n)\cap C^0(\overline{M^n})$ satisfying the following system of equations.
\begin{equation}
    \triangle\varphi+\lambda\varphi=0\quad\text{and}\quad\left.\varphi\right|_{\partial M}=0\,.
\end{equation}

\begin{theorem}\label{Arbitrary-Domain-Estimation-Thm}
    Let $M^n$ and $\lambda\in\mathbb{R}$ be as defined in the Dirihlet Eigenvalue Problem.
    Let $D^n_1$ and $D^n_2$ be two disks in $\mathbb{H}^n$ such that
    \begin{equation}\label{Disc-Manifold-Inclusion-Relationship}
        D^n_1\subseteq M^n\subseteq D^n_2
    \end{equation}
    and let $d_1$ and $d_2$ be the diameters of $D^n_1$ and $D^n_2$ respectively. Then
    \begin{description}
      \item[(A)] For $n=2$
        \begin{equation}
            -\frac\kappa4+\left(\frac{\pi}{d_2}\right)^2\leq\lambda_{\min}(M^2)\leq
            -\frac\kappa4+\left(\frac{2\pi}{d_1}\right)^2\,.
        \end{equation}
      \item[(B)] For $n=3$
        \begin{equation}
            -\kappa+\left(\frac{2\pi}{d_2}\right)^2\leq\lambda_{\min}(M^3)\leq
            -\kappa+\left(\frac{2\pi}{d_1}\right)^2\,.
        \end{equation}
      \item[(C)] For $n>3$
        \begin{equation}
            -\kappa\frac{(n-1)^2}{4}+\left(\frac{2\pi}{d_2}\right)^2\leq\lambda_{\min}(M^n)\,.
        \end{equation}
    \end{description}
\end{theorem}

\begin{remark} The proof of the theorem will be split into two steps. First, we develop technique and prove the theorem under the assumption that $M^n=D^n(\delta)\subseteq B^n_\rho$, where $B^n_\rho$ is the ball model of $n$-dimensional hyperbolic space with a constant sectional curvature $\kappa=-1/\rho^2$ and $D^n(\delta)$ is a hyperbolic disc of radius $\delta$ centered at the origin of $B^n_\rho$ (see theorem~\ref{Lower-and-Upper-bound-Theorem}, page ~\pageref{Lower-and-Upper-bound-Theorem}). The final step is to obtain theorem~\ref{Arbitrary-Domain-Estimation-Thm} as the consequence of theorem ~\ref{Lower-and-Upper-bound-Theorem} and Rayleigh's theorem (see page~\pageref{Proof-of-Arbitrary-Estimation} for the proof of theorem~\ref{Arbitrary-Domain-Estimation-Thm}). The next three sections develop the technique necessary for the first step.
\end{remark}


\section{Elementary Geometry Preliminaries}

In this section we introduce the basic tools used in the derivation of the results stated in the introduction and in the previous section. First, we define a function $\omega(x,y)$ that will be used to build all possible radial eigenfunctions. Let $x,y\in\mathbb{R}^{k+1}$ and assume that $|x|\neq|y|$. Define
\begin{equation}\label{definition of omega}
	\omega(x,y)=\left|\frac{|x|^2-|y|^2}{|x-y|^2}\right|\,.
\end{equation}
Observe that for $k=2$, $\omega(x,y)$ turns into the two dimensional Poisson kernel used to solve Dirichlet problem in a disk. Hermann Schwarz, while studying complex analysis, introduced the geometric interpretation of $\omega$ for the case $|x|<|y|$, see \cite{Schwarz} (pp. 359-361) or \cite{Ahlfors} (p. 168). Below we shall see the general version of this geometric interpretation.

\subsection {Geometric Interpretation of $\omega(x,y)$.}

\

Let $x,y\in\mathbb{R}^{k+1}$ and $|x|\neq|y|$. Let $S^k(R)$ denotes the $k$-dimensional sphere of radius $R$ centered at the origin $O$. Then, we define $x^*$ and $y^*$ as follows. If the line defined by $x$ and $y$ is tangent to $S^k(|x|)$, then $x^*=x$. Otherwise, $x^*$ be the point of $S^k(|x|)$ such that $x^*\neq x$ and $x^*, x, y$ are collinear. Similarly, if the line through $x,y$ is tangent to $S^k(|y|)$, then $y^*=y$. Otherwise, $y^*$ be the point of $S^k(|y|)$ such that $y^*\neq y$ and $y^*, x, y$ are collinear.

\begin{proposition}
\begin{equation}\label{GeoInte-2}
    |x-y^*|=|x^*-y|\,.
\end{equation}
\end{proposition}
\begin{proof}
Consider the $k$-dimensional plane passing though the origin and orthogonal to the line($x,y$). Clearly, $x^*$ is the reflection of $x$ with respect to this plane. For the same reason, $y^*$ is the reflection of $y$, and therefore, segment $x^*y$ is the reflection of $xy^*$. Hence, $|x-y^*|=|x^*-y|$.
\end{proof}

\bigskip

Now, using the proposition above, we introduce the following notation:
\begin{equation}\label{GeoInte-3}
    q=|x-y|\quad\text{and}\quad l=|x-y^*|=|x^*-y|\,.
\end{equation}

\begin{theorem}[Geometric Interpretation]\label{Geometric-Interpretation-theorem}
If $\omega(x,y)$ is defined as in \eqref{definition of omega}, then
\begin{equation}\label{GeoInte-4}
    \omega(x,y)=\frac{l}{q}\,.
\end{equation}
The last expression will be referred to as the Geometric Interpretation of $\omega$.
\end{theorem}

\begin{proof} Figure \ref{Geometric_Interpretation} below represents the two dimensional plane defined by three points: $O,x,y$. The segment $MP$ is the tangent chord to the smaller sphere, say, $S^k(|x|)$ at point $x$. Then, clearly, $a=|Mx|=|xP|$ is one half of the length of the chord. Pythagorean theorem implies that

\begin{figure}[h]
  \center\epsfig{figure=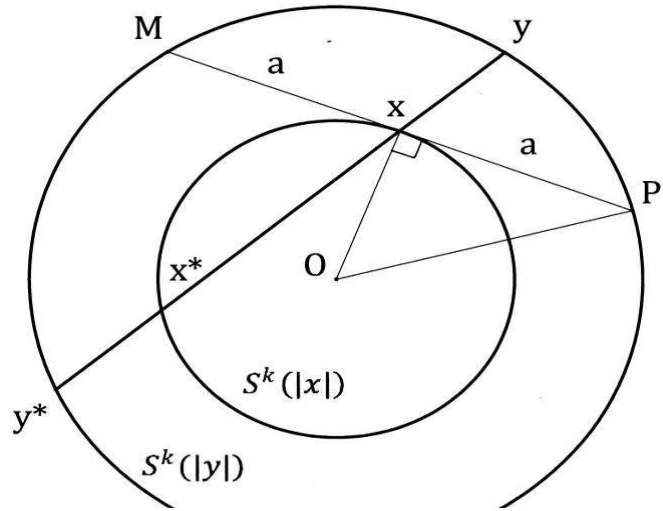, height=7cm, width=9cm}
  \caption{Geometric Interpretation}
  \label{Geometric_Interpretation}
\end{figure}

\begin{equation}\label{GeoInte-5}
    ||x|^2-|y|^2|=a^2\,.
\end{equation}
In addition, $\triangle y^*xM$ is similar to $\triangle Pxy$, which yields
\begin{equation}\label{GeoInte-6}
    a^2=|y^*-x|\cdot |x-y|=lq\,.
\end{equation}
Combining \eqref{GeoInte-5} and \eqref{GeoInte-6}, we may write that
\begin{equation}\label{GeoInte-7}
    \omega(x,y)=\frac{||x|^2-|y|^2|}{|x-y|^2}=\frac{a^2}{q^2}=\frac{lq}{q^2}=\frac{l}{q}\,.
\end{equation}
The same argument applies if the sphere $S^k(|y|)$ is the smaller one. This completes the proof of Theorem~\ref{Geometric-Interpretation-theorem}.
\end{proof}

\subsection {Sphere exchange rule.} \ 

The following lemma describes an important rule that can be used instead of successive application of an inversion and a dilation to integrate over spheres a function that depends only on the distance.

\begin{lemma}[Sphere exchange rule]\label{Basic_lemma}

Let $S^k(r)$ and $S^k(R)$ be
two $k$-dimensional spheres of radii $r$ and $R$ respectively. Let $x,x_1,y,y_1\in \mathbb{R}^{k+1}$ be arbitrary points satisfying $|x|=|x_1|=r$ and $|y|=|y_1|=R$. We assume that $x,y$ are fixed, while $x_1,y_1$ are parameters of integration. If $g:\mathbb{R}\rightarrow \mathbf{C}$ is an integrable complex-valued function on
$[|R-r|,|R+r|]$, then
\begin{equation}\label{ExchaSphere-1}
    R\,^k\cdot \int\limits_{S^k(r)}
    g(|x_1-y|)\,d\,S_{x_1}=r^k\cdot \int\limits_{S^k(R)}
    g(|x-y_1|)\,d\,S_{y_1}\,.
\end{equation}
As a consequence,
\begin{equation}\label{ExchaSphere-2}
    R\,^k\cdot \int\limits_{S^k(r)} g\circ \omega(x_1,y)
    d\,S_{x_1}=r^k\cdot \int\limits_{S^k(R)}
    g\circ\omega(x,y_1)d\,S_{y_1}\,,
\end{equation}
where $g\circ\omega$ denotes the composition of $g$ and $\omega$.


\end{lemma}

\begin{proof}[Proof of Lemma \ref{Basic_lemma}] We prove successively all formulae listed in the Lemma.

\begin{proof}[Proof of \eqref{ExchaSphere-1}.]

We fix $y_0\in S^k(R)$ and define $x_0$ as the following intersection:
\begin{equation}\label{ExchaSphere-4}
    x_0=S^k(R)\cap \text{Ray}(Oy_0)\,.
\end{equation}
Then, let $\widetilde{y}\in S^k(R)$ be a variable point and set
\begin{equation}\label{ExchaSphere-5}
    \widetilde{x}=S^k(r)\cap \text{Ray}(O\widetilde{y})\,.
\end{equation}
The figure below shows the plane defined by Ray($O\widetilde{y}$) and Ray($Oy_0$). The points $x$ and $y$ denoted on the picture below need not be on the cross-sectional plane. Clearly, such a construction yields $\triangle O\widetilde{x}y_0$ and $\triangle O\widetilde{y}x_0$ are congruent, and then,
\begin{equation}\label{ExchaSphere-5.1}
    |\widetilde{y}-x_0|=|\widetilde{x}-y_0|\,.
\end{equation}

\begin{figure}[h]
  \center\epsfig{figure=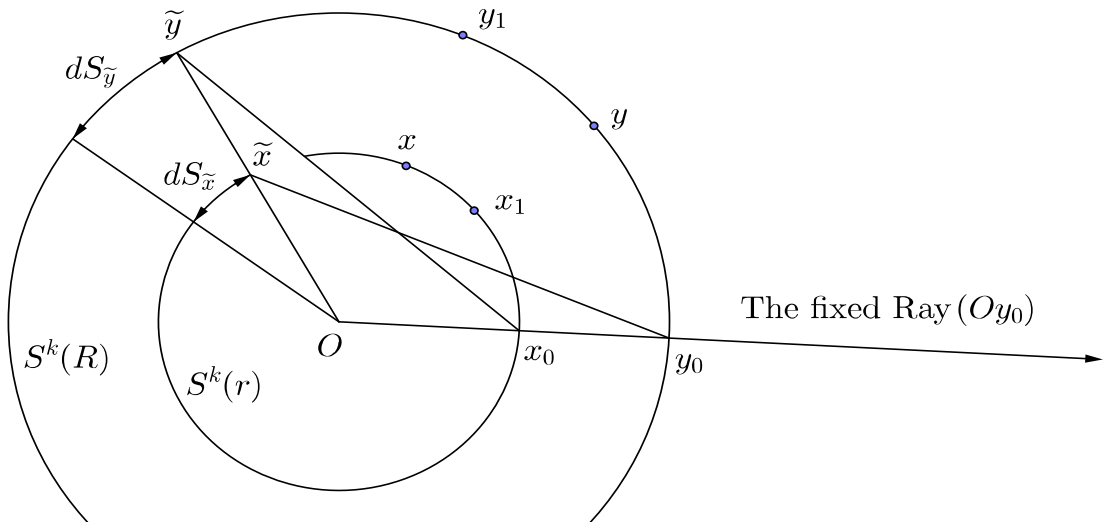, height=7cm, width=15cm}
  \caption{Sphere exchange rule.}\label{Sphere_Exchange}
\end{figure}

\smallskip

Let $dS_{\widetilde{x}}$ be the measure of some spherical infinitesimal neighborhood $U(\widetilde{x})$ around $\widetilde{x}$ and $dS_{\widetilde{y}}$ be the measure of the spherical infinitesimal neighborhood around $\widetilde{y}$ obtained as a dilated image of $U(\widetilde{x})$. This implies that
\begin{equation}\label{ExchaSphere-6}
    dS_{\widetilde{y}}=\left(\frac{R}{r}\right)^k dS_{\widetilde{x}}\,.
\end{equation}
Notice, now, that the function
\begin{equation}\label{ExchaSphere-7}
    F_1(y)=\int\limits_{S^k(r)} g(|x_1-y|)dS_{x_1}
\end{equation}
is invariant under isometries of $\mathbb{R}^{k+1}$ that fix the origin $O$. This is why $F_1(y)=F_1(y_0)$ for all $y,y_0\in S^k(R)$. Then, if we replace in $F_1(y_0)$ the parameter of integration $x_1$ by $\widetilde{x}$, we obtain the following formula
\begin{equation}\label{ExchaSphere-8}
    \int\limits_{S^k(r)} g(|x_1-y|)dS_{x_1}=\int\limits_{S^k(r)} g(|\widetilde{x}-y_0|)dS_{\widetilde{x}}\,.
\end{equation}
Recall that
\begin{equation}\label{ExchaSphere-10}
    |\widetilde{x}-y_0|=|x_0-\widetilde{y}|\quad\text{and}\quad
    dS_{\widetilde{y}}=\left(\frac{R}{r}\right)^k dS_{\widetilde{x}}\,.
\end{equation}
Therefore, by a change of variables and \eqref{ExchaSphere-10},
\begin{equation}\label{ExchaSphere-11}
\begin{split}
    & \int\limits_{S^k(r)} g(|\widetilde{x}-y_0|)dS_{\widetilde{x}}=\left(\frac{r}{R}\right)^k
    \int\limits_{\widetilde{x}\in S^k(r)} g(|x_0-\widetilde{y}|)\left(\frac{R}{r}\right)^k dS_{\widetilde{x}}
    \\& \quad\quad
    =\left(\frac{r}{R}\right)^k
    \int\limits_{\widetilde{x}(\widetilde{y})\in S^k(r)} g(|x_0-\widetilde{y}(\widetilde{x})|) dS_{\widetilde{y}(\widetilde{x})}
    \\& \quad\quad\quad\quad\quad\quad\quad\quad\quad\quad=
    \left(\frac{r}{R}\right)^k
    \int\limits_{\widetilde{y}\in S^k(R)} g(|x_0-\widetilde{y}|)dS_{\widetilde{y}}\,,
\end{split}
\end{equation}
where the last equality follows since
\begin{equation}\label{ExchaSphere-12}
    \widetilde{x}=\widetilde{x}(\widetilde{y})=\frac{r}{R}\cdot \widetilde{y}\in S^k(r)\quad\Leftrightarrow\quad
    \widetilde{y}=\widetilde{y}(\widetilde{x})=\frac{R}{r}\cdot \widetilde{x}\in S^k(R)\,.
\end{equation}
As above, the function
\begin{equation}\label{ExchaSphere-13}
    F_2(x_0)=\int\limits_{S^k(R)} g(|x_0-\widetilde{y}|)dS_{\widetilde{y}}
\end{equation}
is invariant under isometries of $\mathbb{R}^{k+1}$ that fix the origin. Thus,
\begin{equation}\label{ExchaSphere-14}
    \int\limits_{S^k(R)} g(|x_0-\widetilde{y}|)dS_{\widetilde{y}}=
    \int\limits_{S^k(R)} g(|x-\widetilde{y}|)dS_{\widetilde{y}} \quad\forall x,x_0\in S^k(r)\,.
\end{equation}
Finally, gathering all results from the chain \eqref{ExchaSphere-8}, \eqref{ExchaSphere-11}, \eqref{ExchaSphere-14} and changing of notation for the variable of integration, we have
\begin{equation}\label{ExchaSphere-15}
    \int\limits_{S^k(r)} g(|x_1-y|)dS_{x_1}=\left(\frac{r}{R}\right)^k
    \int\limits_{S^k(R)} g(|x-y_1|)dS_{y_1} \,,
\end{equation}
which completes the proof of \eqref{ExchaSphere-1} in Lemma \eqref{Basic_lemma}.
\end{proof}

\medskip

\begin{proof}[Proof of \eqref{ExchaSphere-2}.]

A similar argument is used to prove \eqref{ExchaSphere-2}.

The identity \eqref{ExchaSphere-2} holds, because the function $g\circ
\omega(x,y)=\widetilde{g}(|x-y|)$ is also integrable function of one variable $w=|x-y|\in[|R-r|, |R+r|]$ since \begin{equation}\label{ExchaSphere-16}
    \omega(x,y)=\frac{|R^2-r^2|}{|x-y|^2}
\end{equation}
is continuous as a function of $w=|x-y|$ if $R\neq r$. If $R=r$, observe that $\omega(x,y)\equiv0$ for all $y\neq x$, and then,
\begin{equation}
    \frac{1}{|S^k(r)|}\int\limits_{S^k(r)}g\circ\omega(x_1,y)dS_{x_1}=g(0)
    =\frac{1}{|S^k(R)|}\int\limits_{S^k(R)}g\circ\omega(x,y_1)dS_{y_1}\,.
\end{equation}
Therefore, \eqref{ExchaSphere-2} remains true for $R=r$ as well.
\end{proof}


Therefore, the proof of Lemma \eqref{Basic_lemma} is complete.
\end{proof}


\subsection {A useful property of $l$ and $q$.} \

The next goal is to describe a useful feature of $l$ and $q$ defined above. Recall that if $O$ is the origin, $x, y\in\mathbb{R}^{k+1}$ and $R=|y|>|x|$, then
\begin{equation}\label{Feature-1}
    y^*\in S^k(|y|)\quad\text{such that}\quad y^*,x,y\quad\text{are collinear}\,;
\end{equation}

\begin{equation}\label{Feature-2}
    q(y)=|x-y|\,;\quad l(y)=|x-y^*|\,;\quad\psi=\angle Oxy\,.
\end{equation}
All of the notations are presented on the left Figure~\ref{l,q_exchange} below. Clearly, if $x$ and $R$ are fixed, $l$ and $q$ depend only on $\psi$ since both of the distances $|x-y|$ and $|x-y^*|$ depend only on $\psi, x, R$. Fix some $\psi\in(0,\pi)$. Then we observe that while there is the whole set of points
\begin{equation}\label{Feature-3}
    \overline{y}=\overline{y}(\psi)=\{y\in S^k(R)\, | \,\angle yxO=\psi\}\,,
\end{equation}
we need only one plane defined by $x,O$ and some arbitrary $y\in \overline{y}$ to demonstrate the desired relationship among $l,q$ and $\psi$. This is possible because all the values $l,q,\psi$ are invariant of $y\in \overline{y}$ and can be pictured on the plane passing through $x,O$ and some $y\in \overline{y}(\psi)$. The invariance mentioned implies that $q(\psi)$ and $l(\psi)$ can be defined as follows.
\begin{equation}\label{Feature-5}
    q(\psi)=q(\overline{y}(\psi))=|x-y|\quad\text{for every}\quad y\in \overline{y}\,;
\end{equation}
\begin{equation}\label{Feature-6}
    l(\psi)=l(\overline{y}(\psi))=|x-y^*|\,,
\end{equation}
where
\begin{equation}
   y^*\in S^k(|y|),\,\,\, y\in \overline{y} \quad\text{and}\quad y^*,x,y\quad\text{are collinear}.
\end{equation}
Fix some $y=y(\psi)\in \overline{y}$ and picture $l(\psi), q(\psi)$ on the plane defined by $O,x,y(\psi)$. (See the figure below on the right).
\begin{figure}[h]
  \center\epsfig{figure=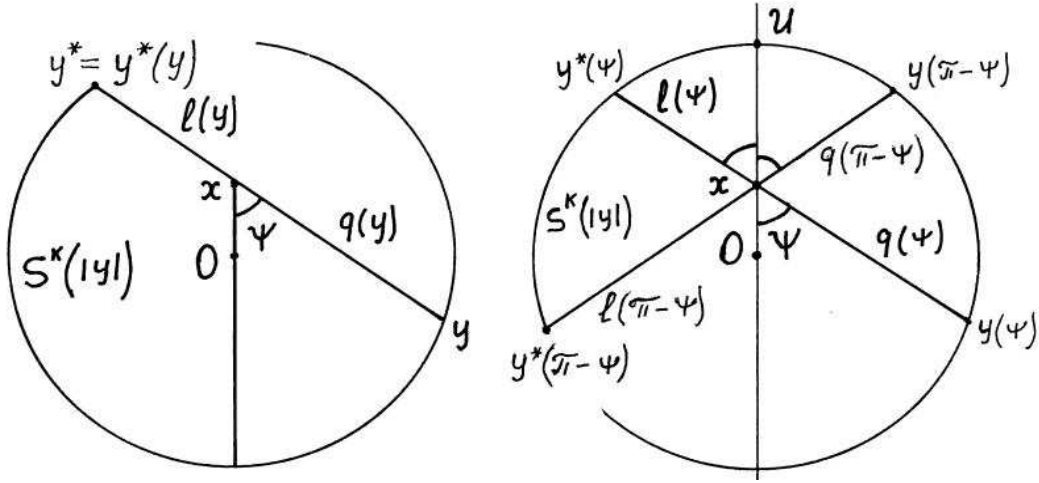, height=6.5cm, width=14cm}
  \caption{Exchanging $l$ and $q$ cumbersome.}\label{l,q_exchange}
\end{figure}

\begin{lemma}[$l, q$ - property]\label{Feature-Feature-lemma}

\begin{equation}\label{Feature-Feature-Lemma-formula}
    l(\psi)=q(\pi-\psi)\quad\text{and}\quad q(\psi)=l(\pi-\psi)\,.
\end{equation}

\end{lemma}

\begin{proof}[Proof of Lemma \ref{Feature-Feature-lemma}] Using the notation described above and the right picture from the Figure~\ref{l,q_exchange} above, we have the following sequence of implications:
\begin{equation}\label{Feature-10}
\begin{split}
    & \angle y(\psi)xO=\psi\quad \Leftrightarrow\quad \angle y^*(\psi)xU=\psi\quad\Leftrightarrow
    \\& \Leftrightarrow\quad \angle y^*(\psi)xO=\pi-\psi=\angle y(\pi-\psi)xO\quad\Rightarrow
    \\& \Rightarrow\quad|x-y^*(\psi)|=|x-y(\pi-\psi)| \quad\Leftrightarrow\quad l(\psi)=q(\pi-\psi)\,.
\end{split}
\end{equation}
The same argument shows that $q(\psi)=l(\pi-\psi)$. This completes the proof of Lemma~\ref{Feature-Feature-lemma}.
\end{proof}

\begin{corollary} It follows that
\begin{equation}\label{Three_D_Dirichlet-28}
    l\left(\frac{\pi}{2}+\tau\right)=q\left(\frac{\pi}{2}-\tau\right)\quad\text{and}\quad
    l\left(\frac{\pi}{2}-\tau\right)=q\left(\frac{\pi}{2}+\tau\right)\,,
\end{equation}
which yields
\begin{equation}\label{Three_D_Dirichlet-29}
\begin{split}
    & (l+q)\circ\left(\frac{\pi}{2}-\tau\right)=(l+q)\circ\left(\frac{\pi}{2}+\tau\right)
    \\& \text{and}\quad\ln\frac{l}{q}\circ\left(\frac{\pi}{2}-\tau\right)
    =-\ln\frac{l}{q}\circ\left(\frac{\pi}{2}+\tau\right)\,,
\end{split}
\end{equation}
where the symbol $\circ$ denotes the composition of two functions.

\end{corollary}

\subsection {Change of variables.}

\smallskip

The following Lemma describes some change rules important for integration. First let us summarize the notation necessary to state the Lemma.

\textbf{Notation:}
\begin{description}
  \item[$S^k(P,R)=S^k_P(R)$] is the $k$-dimensional sphere of radius $R$ centered \\ at $P\in\mathbb{R}^{k+1}$;
  \item[$S^k(R)=S^k(O,R)=S^k_O(R)$] is the $k$-dimensional sphere of radius $R$ centered at the origin $O$;
  \item[$x,y$] are two fixed points in $\mathbb{R}^{k+1}$, such that $r=|x|<|y|=R$;
  \item[$\Sigma=S^k(x,1)$] is the $k$-dimensional unit sphere centered at $x$;
  \item[$\psi=\angle Oxy$] and $\theta=\pi-\angle xOy$;
  \item[$\widetilde{y}=\Sigma\cap \text{line}(xy)$].
\end{description}

\begin{figure}[h]
  \center\epsfig{figure=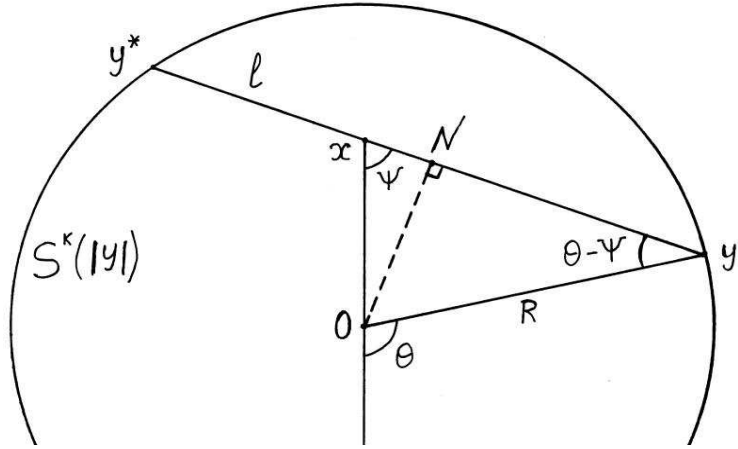, height=6cm, width=10cm}
  \caption{$\theta\leftrightarrow \psi$ Exchange}\label{Theta_Psi_Exchange}
\end{figure}

\begin{lemma}[The integration exchange rules]\label{VarExchange-Lemma}

\begin{equation}\label{VarExchange-1}
  \text{\textbf{(A)}}\quad d\theta=\frac{2q}{l+q}d\psi\quad\text{and}\quad dS_y=\frac{2R}{l+q}q^k d\Sigma_{\widetilde{y}}\,,
\end{equation}
where $dS_y$ and $d\Sigma_{\widetilde{y}}$ are the volume elements of $S^k(R)$ and $\Sigma$ respectively.
\begin{equation}\label{VarExchange-2}
    \text{\textbf{(B)}}\quad\quad\quad\quad\quad\quad\quad  \int\limits_{\Sigma}f(l,q)d\Sigma=\int\limits_{\Sigma}f(q,l)d\Sigma
\end{equation}
for any  complex-valued function $f(l(\psi), q(\psi))$ integrable on $[0,\pi]$.
\end{lemma}

\begin{proof}[Proof of \eqref{VarExchange-1}.]

Using the elementary geometry and the figure above, we observe that $\angle xyO=\theta-\psi$. The law of sines applied to the triangle $\triangle xOy$, gives
\begin{equation}\label{VarExchange-3}
|x|\, \sin(\psi)=R\, \sin(\theta-\psi)\,.
\end{equation}
Differentiation with respect to $\theta$ and $\psi$ yields
\begin{equation}\label{VarExchange-4}
    d\theta=\frac{|x|\, \cos(\psi)+R\, \cos(\theta-\psi)}
    {R\, \cos(\theta-\psi)}\, d\psi\,.
\end{equation}
Again, look at the picture above and observe that if $N$ is the orthogonal projection of the origin $O$ to the chord $yy^*$, then $N$ must be the midpoint for the chord $yy^*$. Therefore,
\begin{equation}\label{VarExchange-4-1}
    |Ny|=R\cos(\theta-\psi)=\frac{l+q}{2}\,,
\end{equation}
which is the denominator in \eqref{VarExchange-4}. Note also that
\begin{equation}\label{VarExchange-4-2}
    q=|x-y|=|xN|+|Ny|=|x|\cos\psi+R\cos(\theta-\psi)\,,
\end{equation}
which is precisely the numerator in \eqref{VarExchange-4}.
Therefore, combining \eqref{VarExchange-4}, \eqref{VarExchange-4-1} and \eqref{VarExchange-4-2}, we have
\begin{equation}\label{VarExchange-5}
    d\theta=\frac{2q}{q+l}\, d\psi\,,
\end{equation}
and then, the first formula in \eqref{VarExchange-1} is complete.

To prove the second formula in \eqref{VarExchange-1}, we introduce some additional notation listed and pictured on Figure~\ref{dS_dSigma_Exchange} below.

\textbf{Additional notation:}
\begin{description}
  \item[$\sigma_k$] is the volume of a $k$-dimensional unit sphere;
  \item[$P_y$] and $P_{\widetilde{y}}$ are the orthogonal projections of $y$ and $\widetilde{y}$ respectively to line $Ox$;
  \item[$H(y)$] and $H(\widetilde{y})$ are the $k$-dimensional hyperplanes passing through $y$ and $\widetilde{y}$ respectively and orthogonal to $Ox$;
  \item[$S^{k-1}(P_y\,, |y P_y |)$]$=H(y)\cap S^k(|y|)$;
  \item[$S^{k-1}(P_{\widetilde{y}}\,, |\widetilde{y} P_{\widetilde{y}}|)$]$=H(\widetilde{y})\cap \Sigma$;
  \item[$dS_y^{k-1}(P_y\,, |yP_y|)$] is the volume element of $S^{k-1}(P_y\,,|yP_y|)$ at point $y$;
  \item[$dS_{\widetilde{y}}^{k-1}(P_{\widetilde{y}}\,, |\widetilde{y} P_{\widetilde{y}}|)$] is the volume element of $S^{k-1}(P_{\widetilde{y}}\,, |\widetilde{y} P_{\widetilde{y}}|)$ at point $\widetilde{y}$.
\end{description}

\begin{figure}[h]
  \center\epsfig{figure=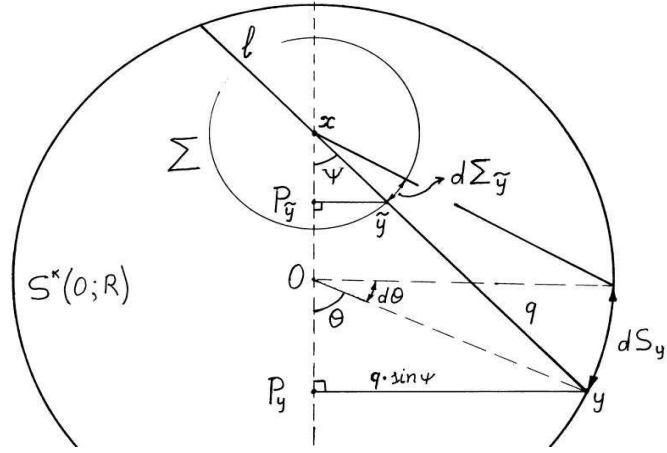, height=6cm, width=9cm}
  \caption{$dS\leftrightarrow d\Sigma$ Exchange}\label{dS_dSigma_Exchange}
\end{figure}

Note first that
\begin{equation}\label{VarExchange-6}
\begin{split}
    & dS_y=Rd\theta\cdot dS_y^{k-1}(P_y\,,|yP_y|)\quad\text{and}
    \\& d\Sigma_{\widetilde{y}}=
    d\psi\cdot dS_{\widetilde{y}}^{k-1}(P_{\widetilde{y}}\,,|\widetilde{y}P_{\widetilde{y}}|) \,.
\end{split}
\end{equation}
On the other hand,
\begin{equation}\label{VarExchange-7}
    dS_y^{k-1}(P_y\,,|yP_y|)=
    q^{k-1}\cdot dS_{\widetilde{y}}^{k-1}(P_{\widetilde{y}}\,,|\widetilde{y}P_{\widetilde{y}}|)\,.
\end{equation}
Therefore, combining \eqref{VarExchange-6}, \eqref{VarExchange-7} and the expression for $d\theta$ from \eqref{VarExchange-1}, we have the following sequence of equalities.
\begin{equation}\label{VarExchange-8}
\begin{split}
    dS_y
    & =Rd\theta\cdot dS_y^{k-1}(P_y\,,|yP_y|)=Rd\theta\cdot q^{k-1}\cdot dS_{\widetilde{y}}^{k-1}(P_{\widetilde{y}}\,,|\widetilde{y}P_{\widetilde{y}}|)
    \\& =\frac{2R}{l+q}q^{k}d\psi\cdot dS_{\widetilde{y}}^{k-1}(P_{\widetilde{y}}\,,|\widetilde{y}P_{\widetilde{y}}|)
    =\frac{2R}{l+q}q^{k}d\Sigma_{\widetilde{y}}\,,
\end{split}
\end{equation}
which completes the proof of \eqref{VarExchange-1}.
\end{proof}

\begin{proof}[Proof of \eqref{VarExchange-2}.]

Note that the function $f(l,q)=f(l(\psi), q(\psi))$ depends only on the angle $\psi$ pictured above. Therefore, if we introduce
\begin{equation}\label{VarExchange-9}
    \Sigma^{k-1}(\psi)=\{\widetilde{y}\in\Sigma\mid\angle \widetilde{y}xO=\psi\}\,,
\end{equation}
we may write
\begin{equation}\label{VarExchange-10}
\begin{split}
    \int\limits_{\widetilde{y}\in\Sigma}f(l,q)d\Sigma_{\widetilde{y}}
    & =\int\limits_0^\pi d\psi\int\limits_{\widetilde{y}\in\Sigma^{k-1}(\psi)}f(l,q)d\Sigma_{\widetilde{y}}^{k-1}(\psi)
    \\& =\int\limits_0^\pi f(l,q)(\sin\psi)^{k-1}\sigma_{k-1}d\psi\,,
\end{split}
\end{equation}
since $f(l,q)$ remains constant while $\widetilde{y}\in\Sigma^{k-1}(\psi)$ and $|\Sigma^{k-1}(\psi)|=\sigma_{k-1}(\sin\psi)^{k-1}$ is the volume of $\Sigma^{k-1}(\psi)$.
 Using \eqref{VarExchange-10}, then  \eqref{Feature-Feature-Lemma-formula} from Lemma \eqref{Feature-Feature-lemma}, p.~\pageref{Feature-Feature-lemma}, and the following change of variables $\widetilde{\psi}=\pi-\psi$, we have the following sequence of equalities.
\begin{equation}\label{VarExchange-11}
\begin{split}
    & \int\limits_{\Sigma}f(l,q)d\Sigma=\int\limits_{\psi=0}^{\psi=\pi} f(l(\psi),q(\psi))\sigma_{k-1}(\sin\psi)^{k-1}d\psi
    \\& =\int\limits_0^\pi f(q(\pi-\psi), l(\pi-\psi))\sigma_{k-1}(\sin\psi)^{k-1}d\psi
    \\& =-\int\limits_\pi^0 f(q(\widetilde{\psi}), l(\widetilde{\psi}))\sigma_{k-1}(\sin\widetilde{\psi})^{k-1}d\widetilde{\psi}
    \\& =\int\limits_{\widetilde{\psi}=0}^{\widetilde{\psi}=\pi} f(q(\widetilde{\psi}), l(\widetilde{\psi})) \sigma_{k-1}(\sin\widetilde{\psi})^{k-1}d\widetilde{\psi}=\int\limits_\Sigma f(q,l)d\Sigma\,,
\end{split}
\end{equation}
which completes the proof of \eqref{VarExchange-2}, and the proof of Lemma \eqref{VarExchange-Lemma}.
\end{proof}

\subsection {The Differentiation of $l/q$ and $l+q$ with respect to $\psi$.}

\begin{lemma}\label{l/q-differentiation-lemma}

\begin{equation}\label{Differention-1}
    \frac{d\omega}{d\psi}=\frac{d}{d\psi}\left(\frac{l}{q}\right)=\frac{l}{q}\cdot\frac{4r\sin\psi}{l+q}\,,
\end{equation}
where $q,l, \psi$ were defined in \eqref{Feature-2} and $\theta=\pi-\angle xOy$.
\end{lemma}

\begin{proof}[Proof of Lemma \ref{l/q-differentiation-lemma}] Recall that for $R=|y|>|x|=r$,
\begin{equation}\label{Differention-2}
    \omega(x,y)=\frac{R^2-r^2}{R^2+r^2+2Rr\cos\theta}=\frac{l}{q}\,,
\end{equation}
and then, since $R^2-r^2=lq$, $R^2+r^2+2Rr\cos\theta=q^2$ and ${R\sin\theta=q\sin\psi}$, the direct computation yields
\begin{equation}\label{Differention-3}
    \frac{d\omega}{d\theta}=\frac{2rl\sin\psi}{q^2}\,.
\end{equation}
Therefore, using \eqref{Differention-3} and the expression for $d\theta$ from \eqref{VarExchange-1}, p.~\pageref{VarExchange-1}, we have
\begin{equation}\label{Differention-4}
    \frac{d\omega}{d\psi}=\frac{d\omega}{d\theta}\frac{d\theta}{d\psi}
    =\frac{2rl\sin\psi}{q^2}\cdot\frac{2q}{l+q}
    =\frac{l}{q}\cdot\frac{4r\sin\psi}{l+q}\,,
\end{equation}
which is precisely what was stated in Lemma \eqref{l/q-differentiation-lemma}.
\end{proof}

\begin{corollary}
\begin{equation}\label{Substitution-for-d-psi}
    d\psi=\frac{l+q}{4r\sin\psi}d\ln\omega\,.
\end{equation}
\end{corollary}

\begin{proof}
Formula \eqref{Substitution-for-d-psi} is the direct consequence of~\eqref{Differention-1}.
\end{proof}

\begin{lemma}\label{(l+q)-Differentiation-Lemma}
\begin{equation}\label{(l+q)-Differentiation-formula}
    \frac{d(l+q)}{d\psi}=-\frac{2|x|^2\sin(2\psi)}{l+q}\,.
\end{equation}
\end{lemma}

\begin{proof}
 Using Figure~\ref{Theta_Psi_Exchange}, p.~\pageref{Theta_Psi_Exchange} and the Pythagorean theorem we can observe that
 $$\left(\frac{l+q}{2}\right)^2=\rho^2-|x|^2\sin^2\psi\,,$$
 which leads to~\eqref{(l+q)-Differentiation-formula}. This completes the proof of Lemma~\ref{(l+q)-Differentiation-Lemma}.
\end{proof}

\section{Notations and Hyperbolic Geometry Preliminaries.}

For the future reference, let us introduce the following notation. Let $H_\rho^n$ and $B^n_\rho$ be the half-space model and the ball model, respectively, of a hyperbolic $n-$dimensional space with a constant sectional curvature $\kappa=1/{\rho^2}<0$. Recall that the half-space model $H_\rho^n$ consists of the open half-space of points $(x_1, ..., x_{n-1}, t)$ in $R^n$ for all $t>0$ and the metric is given by $(\rho/t)|ds|$, where $|ds|$ is the Euclidean distance element. The ball model $B^n_\rho$ consists of the open unit ball $|X|^2+T^2<\rho^2, (X,T)=(X_1, ..., X_{n-1}, T)$ in $R^n$, and the metric for this model is given by $2\rho^2|ds|^2/(\rho^2-|X|^2-T^2)$.

Recall also that in a hyperbolic space a Laplacian can be represent as
\begin{equation}\label{Uniq-6}
\begin{split}
    \triangle
    & =\frac{1}{\rho^2}\left[ t^2\left( \frac{\partial^2}{\partial x_1^2}+...+\frac{\partial^2}{\partial x_{n-1}^2} + \frac{\partial^2}{\partial t^2} \right)-(n-2)t\frac{\partial}{\partial t} \right]
    \\& =\frac{\partial^2}{\partial r^2}+
    \frac{n-1}{\rho}\coth\left(\frac{r}{\rho}\right)\frac{\partial}{\partial r}+\triangle_{S(0,r)}\,,
\end{split}
\end{equation}
where the first expression is the hyperbolic Laplacian expressed by using Euclidean rectangular coordinates in the upper half-space model and the second expression represents  the hyperbolic Laplacian expressed in the geodesic hyperbolic polar coordinates. Here $\triangle_{S(0,r)}$ is the Laplacian on the geodesic sphere of a hyperbolic radius $r$ about the origin.

The next step is to show that $\omega^\alpha$ as an eigenfunction of the Hyperbolic Laplacian in $B^n_{\rho}$.

\begin{proposition}\label{eigenfunction-omega-proposition}
Let $u$ be any point of $S(\rho)$ and $m=(X,T)\in B(\rho)$, where $S(O,\rho)$ and $B(\rho)$ are the Euclidean sphere and ball, respectively, both of the same radius $\rho$ centered at the origin $O$. Let $k=n-1$ and
\begin{equation}\label{Uniq-11}
    \omega=\omega(u,m)=\frac{\rho^2-\eta^2}{|u-m|^2}=\frac{|u|^2-|m|^2}{|u-m|^2}\,.
\end{equation}
Then
\begin{equation}\label{Omega-as-Eigenfunction}
    \triangle_m \omega^\alpha+\frac{\alpha k-\alpha^2}{\rho^2}\, \omega^\alpha=0\,.
\end{equation}
\end{proposition}

\begin{proof}[Proof of the Proposition.]
Notice that
\begin{equation}\label{Uniq-12}
    \triangle t^\alpha+\frac{\alpha k-\alpha^2}{\rho^2} t^\alpha=0
\end{equation}
and the relationship between $H^n_\rho$ and $B^n_\rho$ is given by Cayley Transform \\
 $K:B^n_\rho\rightarrow H^n_\rho$ expressed by the following formulae
\begin{equation}\label{Uniq-13}
\begin{split}
& x=\frac{2\rho X}{|X|^2+(T-\rho)^2}
\\& t=\frac{\rho^2-|X|^2-T^2}{|X|^2+(T-\rho)^2}=\frac{\rho^2-\eta^2}{|u-m|^2}=\omega(u,m)\,,
\end{split}
\end{equation}
where $u=(0,\rho)\in\partial B(O,\rho)=S(O,\rho)$. Since Cayley transform is an isometry, an orthogonal system of geodesics defining the Laplacian in $B^n_\rho$ is mapped isometrically to an orthogonal system of geodesics in $H^n_\rho$ and therefore,
\begin{equation}\label{Uniq-14}
    t^\alpha=\omega^\alpha(u,m)\quad \text{and}\quad \triangle 								 		 t^\alpha=\triangle_m\omega^\alpha(u,m)\,.
\end{equation}
Therefore, equation \eqref{Omega-as-Eigenfunction} is precisely equation \eqref{Uniq-12} written in $B^n_\rho$, which completes the proof of Proposition~\ref{eigenfunction-omega-proposition}.
\end{proof}



\subsection {Explicit representation of radial eigenfunctions.}

In this subsection we shall see that every radial eigenfunction in $B^{k+1}_{\rho}$ depending only on the distance from the origin has an explicit integral representation.

\begin{definition}\label{sharp-definition} Recall that $B^{k+1}_\rho$ is the ball model of the hyperbolic space~$H^{k+1}_\rho$ with the sectional curvature $\kappa=-1/\rho^2$ and let $B^{k+1}(O,\rho)$ be the Euclidean ball of radius $\rho$ centered at the origin and represents the ball model $B^{k+1}_\rho$. Suppose that $f$ is a function on $B^{k+1}_\rho$. We define its radialization about the origin $O$, written $f^\sharp_O (m)$, by setting
\begin{equation}\label{Radialization-Definition}
    f^\sharp_O(m)=\frac{1}{|S^k(|m|)|}\int\limits_{S^k(|m|)} f(m_1) dS_{m_1}\,,
\end{equation}
where the integration is considered with respect to the measure on $S^k(|m|)$ induced by the Euclidean metric of $\mathbb{R}^n\supset B^{k+1}(O,\rho)$; $|m|$ is the Euclidean distance between the origin $O$ and a point $m\in B^{k+1}_\rho$; $|S^k(|m|)|$ is the Euclidean volume of $S^k(|m|)$.
\end{definition}

The following lemma is a consequence of the uniqueness of Haar measure.

\begin{lemma}\label{Laplacian-Commutes-lemma}
\begin{equation}\label{Laplacian-Commutes}
    \triangle_m f^\sharp_O(m)=\frac{1}{|S^k(|m|)|}\int\limits_{S^k(|m|)}\triangle_{m_1} f(m_1) dS_{m_1}\,.
\end{equation}
\end{lemma}

The next step is to obtain the explicit representation for radial eigenfunctions.

\begin{definition} Let $u\in S^k(\rho)$ be a fixed point and $|m_1|=|m|=\eta<\rho=|u|$. Then let us define
\begin{equation}\label{Radialization-of-omega}
    V_\alpha(\eta)=(\omega_O^\alpha)^\sharp(m)=
    \frac{1}{|S(\eta)|}\int\limits_{S(\eta)}\omega^\alpha(u,m_1)dS_{m_1}\,,
\end{equation}
where $\alpha$ is a complex number. Thus, $V_\alpha(\eta)$ is the radialization of $\omega^\alpha(u,m)$ about the origin.
\end{definition}




\begin{theorem}\label{Radial-func-repre-Thoerem} Let $r$ be the hyperbolic distance between the origin $O$ and $S(\eta)$. Then, the following function
\begin{equation}\label{Radial-func-presentation}
    \varphi_\mu(r)= V_\alpha(\eta(r))= V_\alpha\left(\rho\tanh\left(\frac{r}{2\rho}\right)\right)\,,
\end{equation}
where $\mu=(\alpha k-\alpha^2)/\rho^2$, is the unique radial eigenfunction assuming the value $1$ at the origin and corresponding to an eigenvalue $\mu$, i.e.,

\begin{equation}\label{Equation-for-Radi-Eigenfunction}
    \triangle\varphi_\mu(r)+\mu\varphi_\mu(r)=0.
\end{equation}
\end{theorem}

\begin{proof}
    Recall that $\eta=|m|$ is the Euclidean distance between $m=(X,T)\in B(\rho)$ and the origin, while $r=r(m)=r(\eta)$ is the hyperbolic distance between the origin and $m$. Therefore, the relationship between $r$ and $\eta$ is
\begin{equation}\label{Eucli-Hype-Coordinates-relationship}
    r=\rho\ln\frac{\rho+\eta}{\rho-\eta}\quad\text{or}\quad\eta=\rho\tanh\left(\frac{r}{2\rho}\right)\,,
\end{equation}
which justifies the last expression in \eqref{Radial-func-presentation}.

Recall also that according to \eqref{Omega-as-Eigenfunction}, $\omega^\alpha$ is the eigenfunction of the hyperbolic Laplacian with the eigenvalue $(\alpha k-\alpha^2)/\rho^2$. Therefore, according to Lemma \ref{Laplacian-Commutes-lemma}, p.~\pageref{Laplacian-Commutes-lemma}, the radialization of $\omega^\alpha$ defined in \eqref{Radialization-of-omega} is also an eigenfunction with the same eigenvalue. Uniqueness of the radial eigenfunction $\varphi_\mu(r)$ assuming the value $1$ at the origin follows from the procedure described in \cite{Olver}, pp.~148-153 or in \cite{Chavel}, p.~272. Observing that $\varphi_\mu(0)=V_{\alpha}(0)=1$ completes the proof of Theorem~\ref{Radial-func-repre-Thoerem}.

\end{proof}


\begin{proposition}\label{Statement-A'-Proposition} Let $\alpha, \beta$ be complex numbers such that $\alpha+\beta=k$; let $\rho, \eta$ be two positive numbers such that $\rho\neq\eta$; let $m\in S^k(\eta), u\in S^k(\rho)$. Then
\begin{equation}\label{Uniq-27}
   \int\limits_{S^k(\eta)}\omega^\alpha(m,u)dS_m=
   \int\limits_{S^k(\eta)}\omega^\beta(m,u)dS_m
\end{equation}
\end{proposition}

\begin{proof}[Proof of Proposition \ref{Statement-A'-Proposition}.] Let us observe that the equivalent form of the identity in \eqref{Uniq-27} can be written as $V_\alpha(|m|)=V_\beta(|m|)$, where $V$ was defined in \eqref{Radialization-of-omega}. Note also that $V_\alpha(|m|)$ as well as $V_\beta(|m|)$, according to \eqref{Laplacian-Commutes}, are the radial eigenfunctions of the Hyperbolic Laplacian with the same eigenvalue
\begin{equation}\label{Uniq-28}
   \lambda=\frac{\alpha(\alpha-k)}{\rho^2}=\frac{\beta(\beta-k)}{\rho^2}\,,
\end{equation}
since $\alpha+\beta=k$. In addition, $V_\alpha(0)=V_\beta(0)=1$. According to Theorem~\ref{Radial-func-repre-Thoerem}, p.~\pageref{Radial-func-repre-Thoerem}, $V_\alpha(|m|)=V_\beta(|m|)$ for all $|m|=\eta<\rho$, which is equivalent to \eqref{Uniq-27} for all $\eta<\rho$. To see that \eqref{Uniq-27} remains true for $\eta>\rho$, apply Lemma~\ref{Basic_lemma} from p.~\pageref{Basic_lemma}. This completes the proof of Proposition~\ref{Statement-A'-Proposition}.
\end{proof}

\begin{corollary}\label{imaginary-part-vanishing-corollary}
If $\alpha=k/2+ib$, then
\begin{equation}\label{imaginary-part-vanishing}
    \int\limits_{S^k}\omega^\alpha dS_y=
    \int\limits_{S^k}\omega^{k/2}\cos(b\ln\omega) dS_y\,.
\end{equation}
\end{corollary}




\section {Radial eigenfunctions vanishing \\ at some finite point.}

In this section we describe the radial eigenfunctions corresponding to real eigenvalues and vanishing at some finite radius $r$. We obtain also all radial eigenfunctions together with their eigenvalues for the Dirichlet Eigenvalue Problem in a hyperbolic 3-dimensional disk.

\subsection {Representations of a radial eigenfunction \\ vanishing  at some finite point.}

We shall see here the explicit representation of a radial eigenfunction corresponding to a real eigenvalue and vanishing at a finite radius.

\begin{theorem}\label{Radial-Vanishing-EF-Theorem} Let $\lambda$ be real. If $\varphi_\lambda(r)$ is a radial eigenfunction vanishing at some $r<\infty$ then $\lambda>-\kappa k^2/4$ or equivalently,
\begin{equation}\label{Real-Lambda-and-alpha-b-relation}
   \lambda=\frac{\alpha k-\alpha^2}{\rho^2}\quad\text{with}\quad \alpha=\frac k2+ib\quad\text{and}\quad b\neq0\,.
\end{equation}
Moreover,
\begin{equation}\label{Radial-Vanishing-EF}
    \varphi_\lambda(r(\eta))=\frac{1}{|S(\eta)|}\int\limits_{S(\eta)}\omega^{k/2\pm ib}(u,m)dS_m\,,
\end{equation}
where $b=\sqrt{\frac{\lambda}{-\kappa}-\frac{k^2}{4}}>0$
or, equivalently,
\begin{equation}\label{Radial-Vanishing-EF-cos}
    \varphi_\lambda(r(\eta))=\frac{1}{|S(\eta)|}\int\limits_{S(\eta)}\omega^{k/2}\cos\left( \sqrt{\frac{\lambda}{-\kappa}-\frac{k^2}{4}}\ln\omega \right) dS_m\,,
\end{equation}
where $\rho=|u|$, $\eta=|m|$, $r(\eta)=\rho\ln\left[(\rho+\eta)/(\rho-\eta)\right]$, $\theta=\pi-\widehat{uOm}$ and
\begin{equation}\label{Three_D_Dirichlet-2}
    \omega=\omega(u,m)=\frac{\rho^2-\eta^2}{|u-m|^2}=
    \frac{\rho^2-\eta^2}{\rho^2+\eta^2+2\eta\rho\cos\theta}=\omega(\eta,\theta)\,.
\end{equation}
\end{theorem}

\begin{proof} Let us assume first that $\varphi_\lambda(r)=0$ for some finite $r$. Recall that according to Theorem~\ref{Radial-func-repre-Thoerem} from p.~\pageref{Radial-func-repre-Thoerem}, the unique radial eigenfunction with an eigenvalue $\lambda\in\mathbb{C}$ is
\begin{equation}\label{Three_D_Dirichlet-3}
    \varphi_\lambda(r)=V_\alpha(\eta(r))=\frac{1}{|S(\eta)|}\int\limits_{S(\eta)}\omega^\alpha dS_m\,,
\end{equation}
where
\begin{equation}\label{Three_D_Dirichlet-4}
    |m|=\eta(r)=\rho\tanh\left(\frac{r}{2\rho}\right)\quad\text{and}\quad\lambda=
    \frac{\alpha k-\alpha^2}{\rho^2}=-\kappa(\alpha k-\alpha^2)\,.
\end{equation}
Let $\alpha=a+ib$. Then
\begin{equation}\label{Three_D_Dirichlet-5}
    \lambda=-\kappa\left( (a+ib)k-(a+ib)^2\right)=-\kappa\left[ak+b^2-a^2+ib(k-2a)\right]\,.
\end{equation}
Therefore, $\lambda$ is real if and only if
\begin{equation}\label{Three_D_Dirichlet-6}
    \alpha\in\gimel
    =\{(a+ib)|\,\, a=k/2\,\, \text{or}\,\, b=0\}
\end{equation}
or equivalently, if
\begin{equation}\label{Iff-conditions-for-Lambda-real}
    \alpha=\frac k2\pm s\quad\text{or}\quad\alpha=\frac k2\pm ib\,,\quad\text{where}\,\,s,b\in\mathbb{R}\,.
\end{equation}
It is clear that $\gimel$ is mapped to the real line $\lambda$ as it is pictured on Figure~\ref{Cross-Image} below.

\begin{figure}[h]
  \center\epsfig{figure=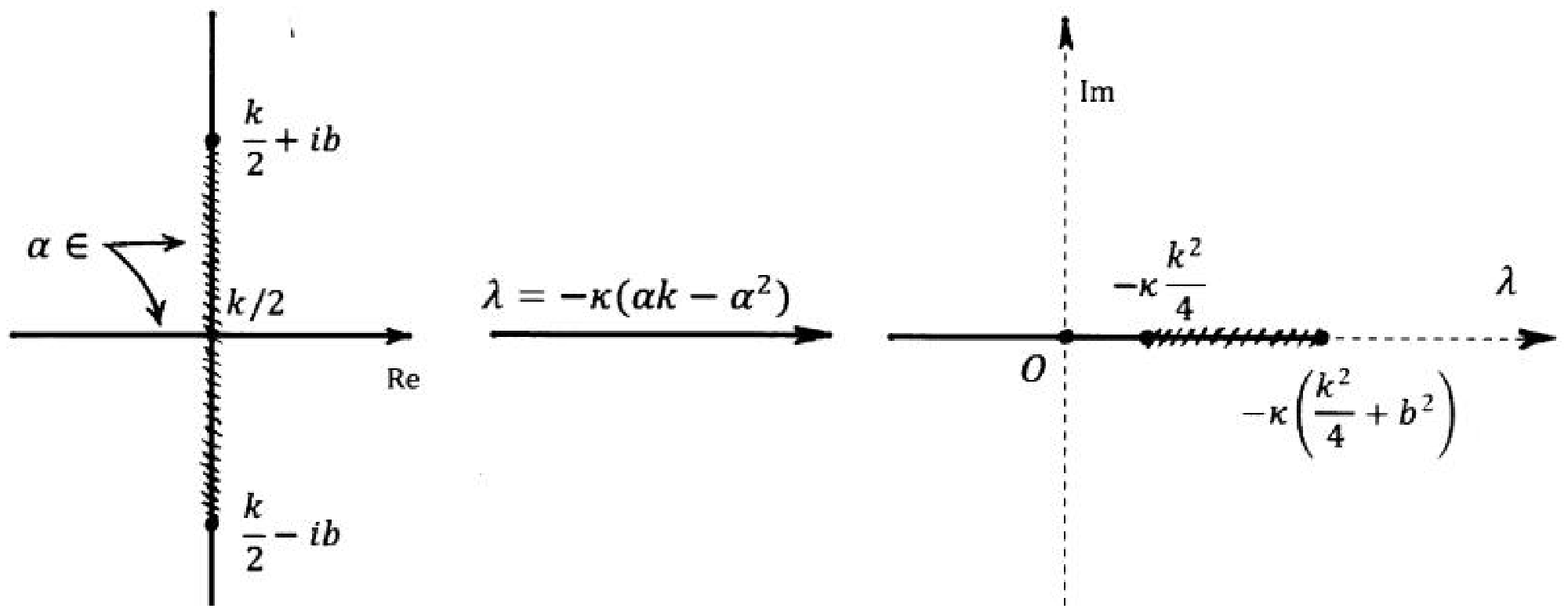, height=6.5cm, width=15cm}
  \caption{The image of the cross $\gimel$.}\label{Cross-Image}
\end{figure}

Using \eqref{Three_D_Dirichlet-5} and \eqref{Three_D_Dirichlet-6} we observe that $\alpha$ is real if and only if
\begin{equation}\label{Three_D_Dirichlet-7}
    \lambda\in\mathbb{R}\quad\text{and}\quad\lambda\leq-\kappa k^2/4\,.
\end{equation}
Therefore, for every $\lambda$ satisfying \eqref{Three_D_Dirichlet-7} the radial eigenfunction does not vanish for any finite $r$ because
\begin{equation}\label{Three_D_Dirichlet-8}
    \varphi_\lambda(r)=V_\alpha(\eta(r))=\frac{1}{|S(\eta)|}\int\limits_{S(\eta)}\omega^\alpha dS_m>0
\end{equation}
for every $\eta\in[0,\rho)$ and for every real $\alpha$. Thus, to allow the radial eigenfunction to vanish at a finite $r=\delta$, $\lambda$ must be strictly greater than $-\kappa k^2/4$. Again, using \eqref{Three_D_Dirichlet-5} and \eqref{Three_D_Dirichlet-6} we can observe that
\begin{equation}\label{Three_D_Dirichlet-9}
    \lambda>-\kappa \frac{k^2}{4}\quad
\end{equation}
yields $\eqref{Radial-Vanishing-EF}$. Note also that $\lambda=-\kappa\left(\alpha k-\alpha^2\right)$ and $\alpha=k/2+ib$ imply together that $b=\pm\sqrt{-\lambda/\kappa-k^2/4}$. Formula~\eqref{imaginary-part-vanishing} of Corollary~\ref{imaginary-part-vanishing-corollary}, p.~\pageref{imaginary-part-vanishing-corollary}, applied to \eqref{Radial-Vanishing-EF} leads directly to~\eqref{Radial-Vanishing-EF-cos}. Therefore, for a vanishing radial eigenfunction all the formulae \eqref{Real-Lambda-and-alpha-b-relation}, \eqref{Radial-Vanishing-EF}, \eqref{Radial-Vanishing-EF-cos}, p.~\pageref{Radial-Vanishing-EF-Theorem} as well as $\lambda>-\kappa k^2/4$ must hold.

\end{proof}

\subsection {Geometric representations of a radial eigenfunction vanishing at a finite point.}\

The formulae presented in Theorem \ref{Radial-Vanishing-EF-Theorem}, p.~\pageref{Radial-Vanishing-EF-Theorem}, lead us to the following geometric presentations of radial eigenfunctions. All notations used in the following lemma are pictured on Figure~\ref{Euclidean-Notations-in-HD} below.

\begin{lemma}\label{Lemma-532} Geometric representation of radial eigenfunctions can be described as follows.

\begin{description}
  \item[\textbf{(A)}] A radial eigenfunction can be expressed as
\begin{equation}\label{Three_D_Dirichlet-15}
   \varphi_\lambda(r(\eta))=\frac{1}{|S(\rho)|}
   \int\limits_{S(\rho)}\omega^\alpha dS_u=
   \frac{1}{|S(\rho)|}
   \int\limits_{\Sigma}\left(\frac{l}{q}\right)^\alpha\frac{2\rho}{l+q}q^k d\Sigma_{\widetilde{u}}\,,
\end{equation}
where $\alpha$ is chosen in such a way that $\lambda=(\alpha k-\alpha^2)/\rho^2$.

  \item[\textbf{(B)}] If a radial eigenfunction has a real eigenvalue and vanishes at some finite radius $r_0<\infty$, then
\begin{equation}\label{Three_D_Dirichlet-16}
    \varphi_\lambda(r(\eta))
    =\frac{1}{|S(\rho)|}\int\limits_{S(\rho)}\left(\frac{l}{q}\right)^{k/2}
    \cos\left(b\ln\frac{l}{q}\right)dS_{u} \,,
\end{equation}
where $u\in S^k(\rho)$ is the parameter of integration,
\begin{equation}\label{Three_D_Dirichlet-17}
    \frac{l}{q}=\frac{|m-u|}{|m-u^*|}=\omega(u,m)\quad\text{and}\quad
    b=\pm\sqrt{\frac{\lambda}{-\kappa}-\frac{k^2}{4}}\,.
\end{equation}

\item[\textbf{(C)}] The integral given in \eqref{Three_D_Dirichlet-16} may be transformed to the following form
\begin{equation}\label{Three_D_Dirichlet-18}
    \varphi_\lambda(r(\eta))
    =\frac{4\rho(\rho^2-\eta^2)^{k/2}\sigma_{k-1}}{|S(\rho)|}\int\limits_0^{\pi/2}
    \frac{\sin^{k-1}\psi}{l+q}\cos\left(b\ln\frac{l}{q}\right)d\psi\,,
\end{equation}
which also represent a vanishing radial eigenfunction, with $b$ defined in~\eqref{Three_D_Dirichlet-17}.
\end{description}
\end{lemma}

\begin{remark}
We shall see that the last integral formula can be simplified using integration by parts for $k>1$. In particular, for $k=2$ this integral can be computed explicitly.
\end{remark}

\begin{proof}[Proof of Lemma \ref{Lemma-532}, Statement (A)] Notice that the restrictions $u\in S(\rho)$ and $m\in S(\eta)$ make $\omega^\alpha(u,m)$ a function depending only on the distance between $u$ and $m$, since
\begin{equation}\label{Three_D_Dirichlet-19}
    \omega^\alpha(u,m)=(\rho^2-\eta^2)^\alpha\cdot\frac{1}{|u-m|^{2\alpha}}=g(|u-m|)\,.
\end{equation}
Therefore, we can apply the Lemma~\ref{Basic_lemma}, p.~\pageref{Basic_lemma} to $\omega^\alpha(u,m)$. Let $u\in S(\rho)$ and $m_1\in S(\eta)$ pictured below serve as the parameters of integration, while $u_1\in S^k(\rho)$ is a fixed point. Let $\Sigma$ be the unit sphere centered at point $m$ and $u^*$ is defined as the intersection of $\Sigma$ and segment $mu$; $\psi=\angle Omu$.

\begin{figure}[h]
  \center\epsfig{figure=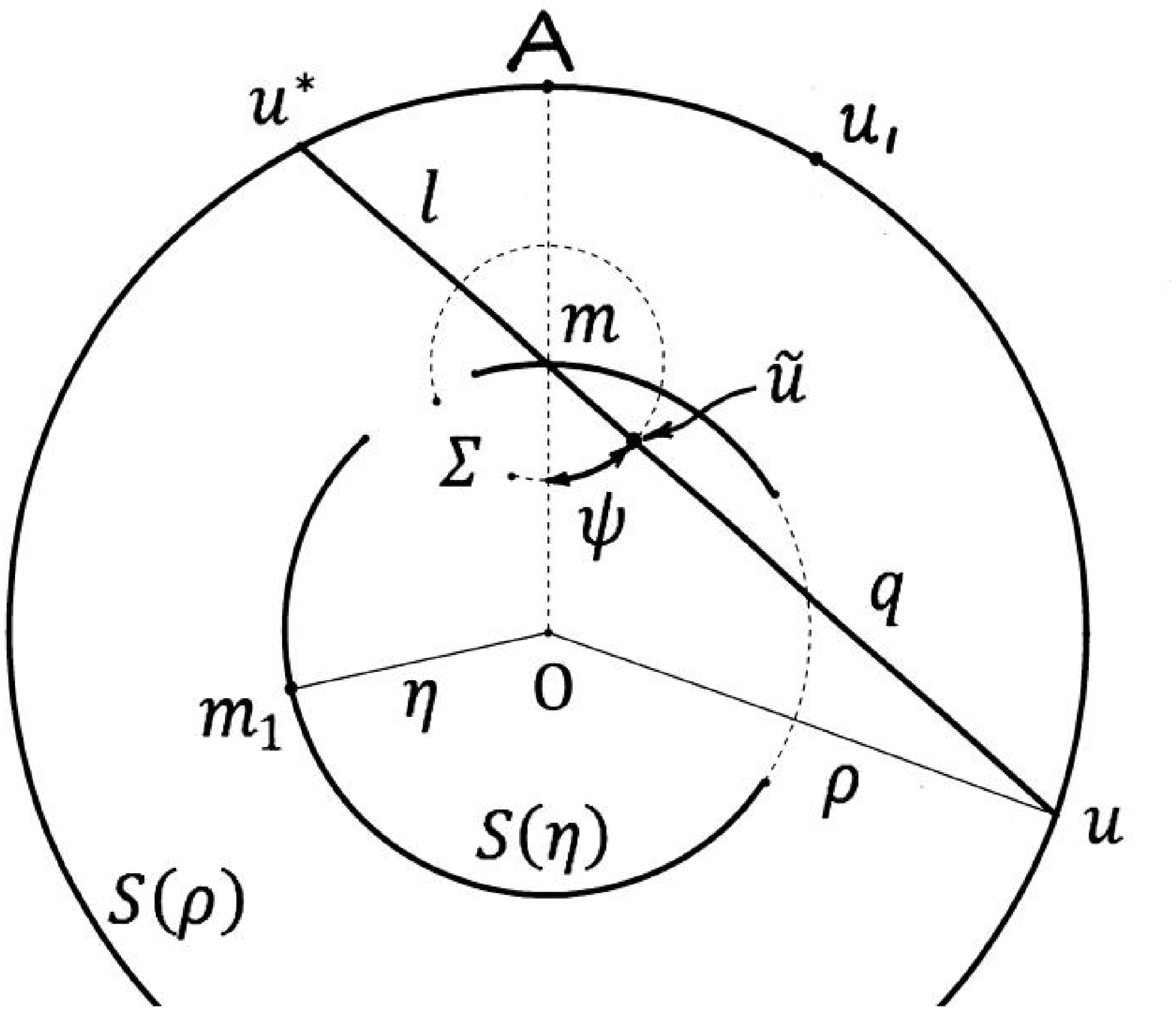, height=6cm, width=9cm}
  \caption{Euclidean variables in the hyperbolic disc.}\label{Euclidean-Notations-in-HD}
\end{figure}

Recall also that according to the geometric interpretation of $\omega$ presented in Theorem~\ref{Geometric-Interpretation-theorem}, p.~\pageref{Geometric-Interpretation-theorem}, we have
\begin{equation}\label{Three_D_Dirichlet-20}
    \omega(u,m)=\frac{|u^*-m|}{|m-u|}=\frac{l}{q}\,,
\end{equation}
where $u^*$ is the intersection of $S(\rho)$ and the line defined by $m$ and $u$.

Using the representation \eqref{Radial-func-presentation} of p.~\pageref{Radial-func-presentation} for a radial eigenfunction,  Lemma~\ref{Basic_lemma} of p.~\pageref{Basic_lemma} and the second exchange rule of~\eqref{VarExchange-1}, p.~\pageref{VarExchange-1}, we obtain the following sequence of equalities
\begin{equation}\label{Three_D_Dirichlet-21}
\begin{split}
     \varphi_\lambda(r)
    & =V_\alpha(\eta(r))=
    \frac{1}{|S(\eta)|}\int\limits_{S(\eta)}\omega^\alpha(u_1,m_1) dS_{m_1}
    \\& =\frac{1}{|S(\rho)|}\int\limits_{S(\rho)}\omega^\alpha(u,m) dS_{u}=
         \frac{1}{|S(\rho)|}\int\limits_\Sigma\left(\frac{l}{q}\right)^\alpha\frac{2\rho}{l+q}q^k d\Sigma_{\widetilde{u}}\,,
\end{split}
\end{equation}
where the third equality is a consequence of Lemma~\ref{Basic_lemma}. This completes the proof of the Statement~(A) of Lemma~\ref{Lemma-532}.
\end{proof}

\begin{proof}[Proof of Lemma \ref{Lemma-532}, Statement (B)]

Using the representation of a radial eigenfunction vanishing at some finite point obtained in Theorem~\ref{Radial-Vanishing-EF-Theorem}, see formula~\eqref{Radial-Vanishing-EF-cos}, p.~\pageref{Radial-Vanishing-EF-cos}, and then Lemma~\ref{Basic_lemma}, p.~\pageref{Basic_lemma}, we have
\begin{equation}\label{Three_D_Dirichlet-22}
\begin{split}
    & \varphi_\lambda(r(\eta))=\frac{1}{|S(\eta)|}\int\limits_{S(\eta)}\omega^{k/2}(u_1,m_1)\cos\left( \sqrt{\frac{\lambda}{-\kappa}-\frac{k^2}{4}}\ln\omega \right) dS_{m_1}
    \\& =\frac{1}{|S(\rho)|}\int\limits_{S(\rho)}\omega^{k/2}(u,m)\cos\left( \sqrt{\frac{\lambda}{-\kappa}-\frac{k^2}{4}}\ln\omega \right) dS_{u}    \,,
\end{split}
\end{equation}
where, according to \eqref{Three_D_Dirichlet-20}, $\omega=l/q$. This completes the proof of Statement~(B).
\end{proof}

\begin{proof}[Proof of Lemma \ref{Lemma-532}, Statement (C)]

 As we saw in Statement (B), a radial eigenfunction vanishing at some finite point and corresponding to a real eigenvalue can be written as
\begin{equation}\label{Three_D_Dirichlet-23}
    \varphi_\lambda(r(\eta))
    =\frac{1}{|S(\rho)|}\int\limits_{S(\rho)}\left(\frac{l}{q}\right)^{k/2}
    \cos\left(b\ln\frac{l}{q}\right)dS_{u} \,,
\end{equation}
where $b=\sqrt{(-\lambda/\kappa)-k^2/4}$ and the variables $l,q,u,m$ were introduced above. If we apply the argument used in the proof of Theorem~\ref{Geometric-Interpretation-theorem}, p.~\pageref{Geometric-Interpretation-theorem}, see formulas~\eqref{GeoInte-5} and~\eqref{GeoInte-6}, we observe that $lq=\rho^2-\eta^2$. According to the second formula of~\eqref{VarExchange-1}, p.~\pageref{VarExchange-1},
\begin{equation}\label{Var-Excha-New-Notation-dS_u}
    dS_u=\frac{2\rho}{l+q}q^k d\Sigma_{\widetilde{u}}\,.
\end{equation}
These two expressions for $lq$ and for $dS_u$ allow us to rewrite the integral formula \eqref{Three_D_Dirichlet-23} for $\varphi_\lambda(r(\eta))$ in the following way.
\begin{equation}\label{Three_D_Dirichlet-24}
\begin{split}
    \varphi_\lambda(r(\eta))
    & =\frac{1}{|S(\rho)|}\int\limits_{\Sigma}\left(\frac{l}{q}\right)^{k/2}
    \cos\left(b\ln\frac{l}{q}\right) \frac{2\rho}{l+q}q^k d\Sigma_{\widetilde{u}} \,,
    \\& =\frac{2\rho(\rho^2-\eta^2)^{k/2}}{|S(\rho)|}\int\limits_{\Sigma}
    \frac{\cos(b\ln l/q)}{l+q}d\Sigma_{\widetilde{u}}\,.
\end{split}
\end{equation}
For the next step we need the following observation. For an angle $\psi\in[0,\pi]$ define $\Sigma(\psi)$ as follows.
\begin{equation}\label{Three_D_Dirichlet-25}
    \Sigma(\psi)=\{\widetilde{u}\in\Sigma \mid \angle Om\widetilde{u}=\psi\}\,.
\end{equation}
It is clear that $\Sigma(\psi)$ is a $(k-1)-$dimensional sphere of radius $\sin\psi$ and for any fixed $\psi\in[0,\pi]$ the last integrand in \eqref{Three_D_Dirichlet-24} does not depend on $\widetilde{u}\in\Sigma(\psi)$. Therefore,
\begin{equation}\label{Three_D_Dirichlet-26}
    \int\limits_{\Sigma}\frac{\cos(b\ln l/q)}{l+q}d\Sigma_{\widetilde{u}}
    =\int\limits_0^{\pi}\frac{\cos(b\ln l/q)}{l+q}|\Sigma(\psi)|d\psi\,,
\end{equation}
where $|\Sigma(\psi)|=\sigma_{k-1}(\sin\psi)^{k-1}$ is the volume of the sphere $\Sigma(\psi)$ and $\sigma_{k-1}$ is the volume of $(k-1)-$dimensional unit sphere. Therefore, using \eqref{Three_D_Dirichlet-26}, we may continue the sequence of equalities in \eqref{Three_D_Dirichlet-24}, which yields
\begin{equation}\label{Three_D_Dirichlet-27}
    \varphi_\lambda(r(\eta))
    =\frac{2\rho(\rho^2-\eta^2)^{k/2}\sigma_{k-1}}{|S(\rho)|}\int\limits_0^\pi
    \frac{\sin^{k-1}\psi}{l+q}\cos\left(b\ln\frac{l}{q}\right)d\psi\,.
\end{equation}
According to \eqref{Three_D_Dirichlet-29}, p.~\pageref{Three_D_Dirichlet-29}, the last integrand is symmetric with respect to $\psi=\pi/2$. Therefore, the integral formula for a radial function presented in \eqref{Three_D_Dirichlet-27} can be written as
\begin{equation}\label{Three_D_Dirichlet-30}
    \varphi_\lambda(r(\eta))
    =\frac{4\rho(\rho^2-\eta^2)^{k/2}\sigma_{k-1}}{|S(\rho)|}\int\limits_0^{\pi/2}
    \frac{\sin^{k-1}\psi}{l+q}\cos\left(b\ln\frac{l}{q}\right)d\psi\,,
\end{equation}
which completes the proof of Statement~(C) and the proof of Lemma~\ref{Lemma-532}.
\end{proof}

\section {Lower and upper bounds for the minimal eigenvalue \\ in a Dirichlet Eigenvalue Problem.}

In this section we obtain the lower and the upper bounds for the minimal positive eigenvalue of a Dirichlet Eigenvalue Problem stated on page~\pageref{Dirichlet-Eigen-General-Chavel}. First, in Theorem~\ref{Lower-and-Upper-bound-Theorem} below, we prove the set of inequalities for $M^n=D^n(\delta)\subseteq B^n_\rho$, which is a hyperbolic disc of radius $\delta$ centered at the origin and considered in the hyperbolic space model $B^n_\rho$, where $n=k+1$. Then, in Theorem~\ref{Arbitrary-Domain-Estimation-Thm}, p.~\pageref{Arbitrary-Domain-Estimation-Thm}, we obtain a set of inequalities for an arbitrary relatively compact and connected domain $\overline{M^n}\subseteq H^n$ with $\partial M^n\neq\emptyset$.

\begin{theorem}\label{Lower-and-Upper-bound-Theorem}

Recall that $B^{k+1}_\rho$ is the ball model of $(k+1)-$dimensional hyperbolic space with a constant sectional curvature $\kappa=-1/\rho^2$ and let $\varphi_\lambda(\upsilon, r)$ satisfies the following conditions:
\begin{equation}\label{Lower_Bound-1}
\left\{
  \begin{array}{ll}
     & \hbox{$\triangle\varphi_\lambda(\upsilon,r)+\lambda \varphi_\lambda(\upsilon,r)=0 \quad
     \text{for every}\,\,\, r\in [0,\delta], \,\,\lambda - \text{real}$;} \\
     & \hbox{$\varphi_\lambda(\upsilon, \delta)=0 \quad \text{for some}\,\, \delta\in(0,\infty)$,}
  \end{array}
\right.
\end{equation}
where $\upsilon$ is a point of the unit $k-$dimensional sphere centered at the origin and $r$ is the geodesic distance between a point in $B^{k+1}_\rho$ and the origin, i.e., $(\upsilon, \delta)$ are the geodesic polar coordinates. Then the following statements hold.

\begin{itemize}
	\item[\textbf{(A)}] If $k=1$, then
\begin{equation}\label{Lower_Bound-2}
    -\kappa\frac{k^2}{4}+\left(\frac{\pi}{2\delta}\right)^2<\lambda_{\min}<
    -\kappa\frac{k^2}{4}+\left(\frac{\pi}{\delta}\right)^2 \,.
\end{equation}

	\item[\textbf{(B)}] If $k\geq 3$, then
\begin{equation}\label{Lower_Bound-4}
    \lambda_{\min} >-\kappa\frac{k^2}{4}+\left(\frac{\pi}{\delta}\right)^2\,.
\end{equation}

	\item[\textbf{(C)}] If $k=2$, then
\begin{equation}\label{Lower_Bound-3}
    \lambda_{\min} =-\kappa+\left(\frac{\pi}{\delta}\right)^2\,.
\end{equation}
Moreover, for $k=2$, all radial eigenvalues can be computed as
\begin{equation}\label{Three_D_Dirichlet-32}
    \lambda=\lambda_j=-\kappa+\left(\frac{\pi j}{\delta}\right)^2 \quad\text{for some}\,\, j=1,2,3,\cdots
\end{equation}
and for each $\lambda_j$ there is a unique radial eigenfunction assuming value one at the origin. Such radial eigenfunction is also computable explicitly as follows.
\begin{equation}\label{Expli-Diri-Solu-3-D}
    \varphi_{\lambda_j}(r)=\frac{\delta(\rho^2-\eta^2)}{2\pi\rho^2 \eta j}\cdot\sin\left(\frac{\pi jr}{\delta}\right)
    =\frac{\delta}{\pi j\rho \sinh(r/\rho)}\cdot\sin\left(\frac{\pi jr}{\delta}\right)\,,
\end{equation}
where $0\leq r\leq\delta<\infty$ and $\eta=\rho\tanh(r/2\rho)$.
	
\end{itemize}

\begin{remark} For $k=1$, a stronger upper bound was obtained by Gage, see~\cite{Gage} or see~\cite{Chavel}, p.~80, but the lower bound given in \eqref{Lower_Bound-2} is new.
\end{remark}
\end{theorem}


\begin{proof}[Proof of theorem \ref{Lower-and-Upper-bound-Theorem}]

According to Theorem 2 of \cite{Chavel}, p.44, the lowest positive Dirichlet eigenvalue in the $n-$disk of radius $\delta$ must have a non-trivial radial eigenfunction. According to \cite{Chavel}, p.272, a non-trivial radial eigenfunction cannot vanish at the origin. Therefore, to study the smallest eigenvalue, it is enough to work only with the eigenfunctions assuming the value one at the origin. Such type of radial eigenfunctions will be used throughout the proof of the theorem.

\begin{proof}[Proof of Statement (A) of Theorem \ref{Lower-and-Upper-bound-Theorem}]

First we derive the lower bound of \eqref{Lower_Bound-2}. Recall that according to Theorem 2 of \cite{Chavel}, p.44, the lowest positive Dirichlet eigenvalue in the $n-$disk of radius $\delta$ must have a non-trivial radial eigenfunction satisfying \eqref{Lower_Bound-1}. We have already seen in \eqref{Three_D_Dirichlet-18}, p.~\pageref{Three_D_Dirichlet-18}, any radial eigenfunction assuming value one at the origin and vanishing at some finite point $r=r_0$, can be expressed as
\begin{equation}\label{Lower_Bound-5}
    \varphi_\lambda(r(\eta))
    =\frac{4\rho(\rho^2-\eta^2)^{k/2}\sigma_{k-1}}{|S^k(\rho)|}\int\limits_0^{\pi/2}
    \frac{\sin^{k-1}\psi}{l+q}\cos\left(b\ln\frac{l}{q}\right)d\psi\,,
\end{equation}
where
\begin{equation}\label{Lower_Bound-6}
    b=\pm\sqrt{\frac{\lambda}{-\kappa}-\frac{k^2}{4}}\,,\,\,\,l=|m-u^*|\,,\,\,\,q=|m-u|\,.
\end{equation}
For reader's convenience the notations $\eta, \rho, l, q, u, u^*$ are shown in Figure~\ref{Maximal-position-for-Omega-figure} below. To obtain the minimal Dirichlet Eigenvalue for problem \eqref{Lower_Bound-1} it is enough to find the minimal $\lambda$ for which $\varphi_\lambda(r(\eta))$ represented by \eqref{Lower_Bound-5} vanishes at $r=\delta$.

\begin{proposition}\label{Maximal-position-for-Omega} If
\begin{equation}\label{Lower_Bound-7}
    f(\eta,\psi)=\left| b\ln\frac{l}{q} \right|=\left| b\ln\frac{\rho^2-\eta^2}{\rho^2+\eta^2+2\rho\eta\cos(\theta(\psi))}\right|
\end{equation}
and
\begin{equation}\label{Lower_Bound-8}
    \Gamma=\{ (\eta,\psi) \mid \eta\in[0,\delta_E]\,\,\, \text{and} \,\,\,\psi\in[0,\pi/2] \}\,,
\end{equation}
then
\begin{equation}\label{Lower_Bound-9}
    \max\limits_{\Gamma}f(\eta,\psi)=\left|\frac{b\delta}{\rho}\right|\,,\quad\text{where}\,\,
    \delta=\delta_H=\rho\ln\frac{\rho+\delta_E}{\rho-\delta_E}\,,
\end{equation}
\end{proposition}
where $\eta=|Om|=r_E=\rho\tanh(r/(2\rho))$ and $\delta_E=\rho\tanh(\delta/(2\rho))$.

\begin{figure}[h]
  \center\epsfig{figure=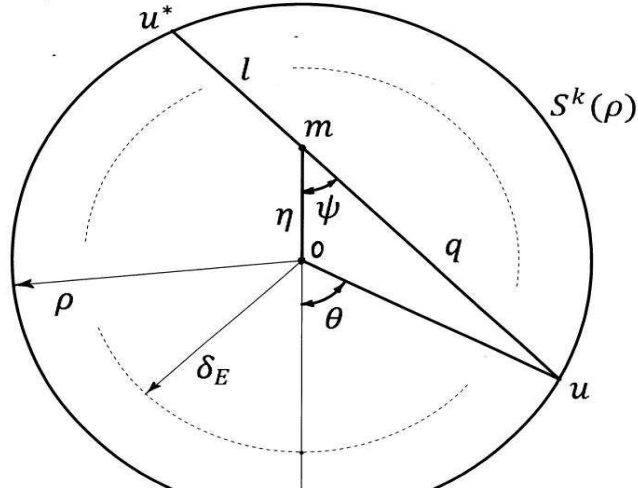, height=6.5cm, width=8.5cm}
  \caption{Euclidean variables affecting $\omega$.}\label{Maximal-position-for-Omega-figure}
\end{figure}

\begin{proof}[Proof of Proposition \ref{Maximal-position-for-Omega}]
 Notice first that
\begin{equation}\label{Lower_Bound-10}
    \frac{d}{d\theta}\left(\frac{l}{q}\right)\geq0\quad\text{for}\,\,\,(\eta, \theta)\in[0,\rho)\times[0,\pi]\,,
\end{equation}
since the denominator in \eqref{Lower_Bound-7} is a strictly decreasing function for all $\theta\in[0,\pi]$ and for $\eta\in(0,\rho)$. Thus,
\begin{equation}\label{Lower_Bound-11}
    \frac{d}{d\psi}\left(\frac{l}{q}\right)=\frac{d(l/q)}{d\theta}\frac{d\theta}{d\psi}\geq0
    \quad\text{for}\,\,\,(\eta, \psi)\in[0,\rho)\times[0,\pi]\,,
\end{equation}
since $\theta$ is increasing whenever $\psi$ is increasing or by \eqref{VarExchange-1}, p.~\pageref{VarExchange-1},
\begin{equation}\label{Lower_Bound-12}
    \frac{d\theta}{d\psi}=\frac{2q}{l+q}>0\quad\text{for all}
    \,\,\,\psi\in[0,\pi]\,\,\,\text{and}\,\,\,\eta\in[0,\rho)\,.
\end{equation}
Therefore, taking into account that
\begin{equation}\label{Lower_Bound-13}
    \frac{l}{q}<1\quad\text{for}\,\,\,(\eta,\psi)\in(0,\rho)\times[0,\pi/2)\,,
\end{equation}
we have
\begin{equation}\label{Lower_Bound-14}
    \frac{df(\eta,\psi)}{d\psi}=\frac{d}{d\psi}\left|b\ln\frac{l}{q}\right|\leq0\quad
    \text{for all}\,\,\,(\eta,\psi)\in[0,\rho)\times[0,\pi/2]\,.
\end{equation}
Thus, for any $\eta\in[0,\rho)$,
\begin{equation}\label{Lower_Bound-15}
    \max\limits_{\psi\in[0,\pi/2]}f(\eta,\psi)=f(\eta,0)=
    \left|\frac{b}{\rho}\cdot\rho\ln\frac{\rho+\eta}{\rho-\eta}\right|=\left|\frac{br}{\rho}\right|\,.
\end{equation}
Hence,
\begin{equation}\label{Lower_Bound-16}
    \max\limits_{\Gamma}f(\eta,\psi)=\max\limits_{\eta\in[0,\delta_E]}f(\eta,0)=
    \max\limits_{\eta\in[0,\delta_E]} \left|\frac{b\cdot r(\eta)}{\rho}\right|
    =\left|\frac{b\delta}{\rho}\right|\,,
\end{equation}
which completes the proof of Proposition~\ref{Maximal-position-for-Omega}.
\end{proof}

Using Proposition \ref{Maximal-position-for-Omega} above, we conclude that
\begin{equation}\label{Lower_Bound-17}
    \left|\frac{b\delta}{\rho}\right|\leq\frac{\pi}{2}\quad\text{implies}\quad
    \left|b\ln\frac{l}{q}\right|_{\Gamma}\leq\frac{\pi}{2}\,,
\end{equation}
which means that the integrand in \eqref{Lower_Bound-5} remains non-negative for all $(\eta,\psi)\in\Gamma$, and therefore, $\varphi_\lambda(r)$ remains positive for all $r\in[0, \delta]$. Thus, because of \eqref{Lower_Bound-17}, the boundary condition in the Dirichlet Eigenvalue Problem \eqref{Lower_Bound-1} will fail. Conversely, if $\varphi(\delta)=0$, then the first inequality in \eqref{Lower_Bound-17} must fail, i.e,
\begin{equation}\label{Lower_Bound-18}
    \left|\frac{b\delta}{\rho}\right|=
    \frac{\delta}{\rho}\sqrt{\frac{\lambda}{-\kappa}-\frac{k^2}{4}}>\frac{\pi}{2},\quad
    \text{with}\,\,\,\kappa=-\frac{1}{\rho^2}
\end{equation}
or, equivalently,
\begin{equation}\label{Lower_Bound-19}
    \lambda>-\kappa\frac{k^2}{4}+\left(\frac{\pi}{2\delta}\right)^2
\end{equation}
is a necessary condition for $\varphi_\lambda(\delta)=0$. This completes the proof of the lower bound in Statement (A). Note that we did not use the assumption that $k=1$ and therefore, this result holds for all $k\geq1$. On the other hand, for $k>1$ we have much stronger statements given in (B) and (C) of Theorem~\ref{Lower-and-Upper-bound-Theorem}, p.~\pageref{Lower-and-Upper-bound-Theorem}. To obtain the upper bound in \eqref{Lower_Bound-2} of the Statement (A), the assumption $k=1$ is essential. Fix $k=1$, $\eta=\delta_E$ and recall that $\sigma_0=2$. Then, from \eqref{Lower_Bound-5}, the value of radial eigenfunction at $r=\delta$ or, equivalently at $r_E=\delta_E$ can be written as
\begin{equation}\label{Upper_Bound-1}
    \varphi_{\lambda}(\delta)=\frac{4\rho}{\pi}(\rho^2-\delta_E^2)^{1/2}\cdot
    \int\limits_0^{\pi/2}\frac{1}{l+q}\cos\left(b\ln\frac{l}{q}\right)d\psi\,.
\end{equation}
Recall from \eqref{Substitution-for-d-psi}, p.~\pageref{Substitution-for-d-psi} that
\begin{equation}\label{Upper_Bound-2-1}
    d\psi=\frac{l+q}{4\eta b\sin\psi}\,\,d\left(b\ln\frac{l}{q}\right)\,.
\end{equation}
If we use the substitution \eqref{Upper_Bound-2-1} for $d\psi$, then the integration by parts applied to \eqref{Upper_Bound-1} and the fact that $l=q$ for $\psi=\pi/2$ yield the following expression for $\varphi_\lambda(\delta)=\varphi(\delta, b^*(\lambda))$.
\begin{equation}\label{Upper_Bound-3}
    \varphi(\delta, b^*(\lambda))= W\left[ \left. \frac{\sin(b^*\ln q/l)}{\sin\psi} \right|_{\psi\rightarrow0}-
    \int\limits_0^{\pi/2}\frac{\cos\psi}{\sin^2\psi}\sin\left( b^*\ln\frac{q}{l}\right)d\psi \right]\,,
\end{equation}
where
\begin{equation}\label{Upper_Bound-4}
    W=\frac{\rho(\rho^2-\delta_E^2)^{1/2}}{\pi\delta_E b^*}\quad\text{and}\quad b^*=|b|=\sqrt{\frac{\lambda}{-\kappa}-\frac{k^2}{4}}\,.
\end{equation}
We are looking for $b_0$ for which $\varphi(\delta, b_0(\lambda_{\min}))=0$, where $\lambda_{\min}$ denotes the minimal eigenvalue. As we saw in the proof of the lower bound, see \eqref{Lower_Bound-17},
\begin{equation}\label{Upper_Bound-5}
    b_1^*\leq\frac{\pi\rho}{2\delta}\quad\text{implies}\quad\varphi(\delta, b_1^*(\lambda_1))>0\,.
\end{equation}
Our goal now is to find $b_2^*>b_1^*$ such that $\varphi(\delta, b_2^*(\lambda_2))<0$. It will follow that  $b_1^*<b_0^*<b_2^*$, since, according to \eqref{Upper_Bound-1} $\varphi(\delta, b^*)$ is $\mathbf{C}^{\infty}$ with respect to $b^*$. For the next step we need the following proposition.

\begin{proposition}\label{EF-runs-under-zero}
\begin{equation}\label{Upper_Bound-6}
    \text{If}\quad b_2^*=\frac{\pi\rho}{\delta}\,,\quad\text{then}\quad\varphi(\delta,b_2^*)
    =\varphi\left(\delta, \frac{\pi\rho}{\delta}\right)<0\,.
\end{equation}
\end{proposition}

\begin{proof}[Proof of Proposition~\ref{EF-runs-under-zero}]

We are going to use formula \eqref{Upper_Bound-3}. Notice that
\begin{equation}\label{Upper_Bound-7}
    \left. \left(b_2^*\ln\frac{q}{l}\right)\right|_{\eta=\delta_E}=
    \frac{\pi\rho}{\delta}\ln\frac{\rho+\delta_E}{\rho-\delta_E}=\pi\,.
\end{equation}
and then, we apply the L'H\^opital's rule to compute all necessary limits in \eqref{Upper_Bound-3}. Recall from \eqref{Upper_Bound-2-1} that
\begin{equation}\label{Upper_Bound-8}
    \left. \frac{d(b^*\ln l/q)}{d\psi}\right|_{\eta=\delta_E}=\frac{4\delta_E b^*\sin\psi}{l+q}\,.
\end{equation}
So, the L'H\^opital's rule together with \eqref{Upper_Bound-8} yields
\begin{equation}\label{Upper_Bound-9}
    \left.\lim\limits_{\psi\rightarrow0}\frac{\sin(b_2^*\ln l/q)}{\sin\psi}\right|_{\eta=\delta_E}=
    \lim\limits_{\psi\rightarrow0}\frac{4\delta_E b_2^*}{l+q}\frac{\cos(b_2^*\ln l/q)\sin\psi}{\cos\psi}=0\,.
\end{equation}
Therefore,
\begin{equation}\label{Upper_Bound-10}
    \varphi(\delta, b_2^*)=-\frac{\rho(\rho^2-\delta_E^2)^{1/2}}{\delta_E\pi b_2^*}
    \int\limits_0^{\pi/2}\frac{\cos\psi}{\sin^2\psi}\sin\left(b_2^*\ln\frac{q}{l}\right)d\psi\,.
\end{equation}
The integral in \eqref{Upper_Bound-10} is well defined, since again, using \eqref{Upper_Bound-8} and the L'H\^opital's rule, we have:
\begin{equation}\label{Upper_Bound-11}
    \left.\lim\limits_{\psi\rightarrow0}\frac{\sin(b_2^*\ln q/l)}{\sin^2\psi}\right|_{\eta=\delta_E}
    =\frac{\pi\delta_E}{\delta}<\infty\,.
\end{equation}
We also see that the integral in \eqref{Upper_Bound-10} is positive, since $q/l$ is a decreasing function, such that
\begin{equation}\label{Upper_Bound-12}
    \frac{\rho+\delta_E}{\rho-\delta_E}=\left.\frac{q}{l}\right|_{\psi=0}\geq\frac{q(\psi)}{l(\psi)}
    \geq\left.\frac{q}{l}\right|_{\psi=\pi/2}=1\,,
\end{equation}
while $0\leq\psi\leq\pi/2$ and $\eta=\delta_E$. Hence,
\begin{equation}\label{Upper_Bound-13}
    b_2^*\ln\frac{\rho+\delta_E}{\rho-\delta_E}=b_2^*\frac{\delta}{\rho}
    =\pi\geq b_2^*\ln\frac{q}{l}\geq0\,,
\end{equation}
which implies that
\begin{equation}\label{Upper_Bound-14}
    \sin\left(b_2^*\ln\frac{q}{l}\right)\geq0\quad\text{for}\quad\psi\in\left[0,\frac{\pi}{2}\right]\,.
\end{equation}
Therefore, the integrand in \eqref{Upper_Bound-10} is strictly positive for all $\psi\in(0, \pi/2)$, and then, $\varphi(\delta, b_2^*)<0$. This completes the proof of the Proposition~\ref{EF-runs-under-zero}.
\end{proof}

Note now that if we choose $b_1^*=\pi\rho/(2\delta)$, then the claim of the Proposition~\ref{EF-runs-under-zero} together with \eqref{Upper_Bound-5} implies that
\begin{equation}\label{Upper_Bound-15}
    b_1^*(\lambda_1)<b_0^*(\lambda_{\min})<b_2^*(\lambda_2)\,,
\end{equation}
where $b^*$ and $\lambda$ are related as usual, by the following formula
\begin{equation}\label{Upper_Bound-16}
    b^*(\lambda)=\sqrt{\frac{\lambda}{-\kappa}-\frac{k^2}{4}}
\end{equation}
and $\lambda_{\min}$ is the minimal eigenvalue, such that
$\varphi(\delta, b_0^*(\lambda_{\min}))=0$.
Therefore,
\begin{equation}\label{Upper_Bound-17}
   \frac{\pi\rho}{2\delta}=b_1^*<b_0^*(\lambda_{\min})=
   \sqrt{\frac{\lambda_{\min}}{-\kappa}-\frac{k^2}{4}}<b_2^*=\frac{\pi\rho}{\delta}\,,
\end{equation}
which leads exactly to what we need, i.e.,
\begin{equation}\label{Upper_Bound-18}
    -\kappa\frac{k^2}{4}+\left(\frac{\pi}{2\delta}\right)^2<\lambda_{\min}<
    -\kappa\frac{k^2}{4}+\left(\frac{\pi}{\delta}\right)^2\,,
\end{equation}
and then, the proof of Statement (A) of Theorem~\ref{Lower-and-Upper-bound-Theorem} is complete.
\end{proof}

\begin{proof}[Proof of Statement (B) of Theorem \ref{Lower-and-Upper-bound-Theorem}]

We shall see here that using the integration by parts we can improve the inequality \eqref{Lower_Bound-19} for $k\geq3$. Recall again from \eqref{Lower_Bound-5}, p.~\pageref{Lower_Bound-5} that a vanishing radial eigenfunction can be expressed as
\begin{equation}\label{Lower_Bound-20}
    \varphi_\lambda(r(\eta))
    =\frac{4\rho(\rho^2-\eta^2)^{k/2}\sigma_{k-1}}{|S^k(\rho)|}\int\limits_0^{\pi/2}
    \frac{\sin^{k-1}\psi}{l+q}\cos\left(b\ln\frac{l}{q}\right)d\psi
\end{equation}
and by \eqref{Substitution-for-d-psi}, p.~\pageref{Substitution-for-d-psi} that
\begin{equation}\label{Lower_Bound-21}
    d\psi=\frac{l+q}{4\eta b\sin\psi}\cdot d\left( b\ln\frac{l}{q}\right)\,.
\end{equation}
If we use the substitution \eqref{Lower_Bound-21} for $d\psi$, then the integration by parts applied to \eqref{Lower_Bound-20} and the fact that $l=q$ for $\psi=\pi/2$ yield the following expression for $\varphi_{\lambda}(r(\eta))$.
\begin{equation}\label{Lower_Bound-22}
    \varphi_\lambda(r(\eta))=\frac{\rho(\rho^2-\eta^2)^{k/2}\sigma_{k-1}}{\eta|S^k(\rho)|\cdot|b|}
    \int\limits_{\psi=0}^{\psi=\pi/2}\sin\left(|b|\ln\frac{q}{l}\right)d\sin^{k-2}\psi\,.
\end{equation}
Notice that $q\geq l$ for $(\eta,\psi)\in[0,\rho)\times[0,\pi/2]$ and $\sin\psi$ is strictly increasing function for $0<\psi<\pi/2$. Hence, the integrand in \eqref{Lower_Bound-22} remains non-negative as long as
\begin{equation}\label{Lower_Bound-23}
    0\leq|b|\ln\frac{q}{l}=\left|b\ln\frac{q}{l}\right|\leq\pi\,.
\end{equation}
Recall from \eqref{Lower_Bound-9}, p.~\pageref{Lower_Bound-9} that
\begin{equation}\label{Lower_Bound-24}
    \max\limits_{\Gamma}\left|b\ln\frac{q}{l}\right|=\left|\frac{b\delta}{\rho}\right|\,,
\end{equation}
where
\begin{equation}\label{Lower_Bound-25}
    \Gamma=\{ (\eta,\psi) \mid \eta\in[0,\delta_E]\,\,\, \text{and} \,\,\,\psi\in[0,\pi/2] \}\,.
\end{equation}
Therefore,
\begin{equation}\label{Lower_Bound-26}
    \left|\frac{b\delta}{\rho}\right|\leq\pi\quad\text{implies}\quad\left|b\ln\frac{q}{l}\right|_{\Gamma}\leq\pi\,.
\end{equation}
Then, using the integral representation \eqref{Lower_Bound-22}, we conclude that $\varphi_\lambda(r)>0$ for all $r\in[0,\delta]$. This means that the boundary condition in the Dirichlet Eigenvalue Problem \eqref{Lower_Bound-1}, p.~\pageref{Lower_Bound-1}, can not be satisfied. Thus, finally,
\begin{equation}\label{Lower_Bound-27}
    \left|\frac{b\delta}{\rho}\right|=\frac{\delta}{\rho}\sqrt{\frac{\lambda}{-\kappa}-\frac{k^2}{4}}>\pi
    \quad\text{with}\,\,\,\kappa=-\frac{1}{\rho^2}\,,
\end{equation}
or, equivalently,
\begin{equation}\label{Lower_Bound-28}
    \lambda>-\kappa\frac{k^2}{4}+\left(\frac{\pi}{\delta}\right)^2
\end{equation}
is necessary condition for $\varphi_\lambda(\delta)=0$ to hold. This completes the proof of Statement~(B) of Theorem~\ref{Lower-and-Upper-bound-Theorem}.
\end{proof}

\begin{proof}[Proof of Statement (C) of Theorem \ref{Lower-and-Upper-bound-Theorem}.]
	Let $k=2$. Again, if we use the substitution \eqref{Lower_Bound-21} for $d\psi$, then the integration by parts applied to \eqref{Lower_Bound-20} and the fact that $l=q$ for $\psi=\pi/2$ yield the following expression for $\varphi_{\lambda}(r(\eta))$.
\begin{equation}\label{exact 3-d solution}
    \varphi_\lambda(r(\eta))=\frac{\rho^2-\eta^2}{2\eta\rho |b|}
    \cdot\sin\left(\frac{|b|r}{\rho}\right)\,,
\end{equation}
where
\begin{equation}
		r=\rho\ln\frac{\rho+\eta}{\rho-\eta}\,,\quad
		|b|=\sqrt{-\frac{\lambda}{\kappa}-1}\quad\text{and}\quad
		\kappa=-\frac{1}{\rho^2}\,.
\end{equation}
Therefore, to satisfy the equation $\varphi_\lambda(\delta)=0$, we have to set
$$\frac{\delta}{\rho}\sqrt{\frac{\lambda_j}{-\kappa}-1}=\pi j\,,\quad j=1,2,3,\cdots\,,$$
which leads to~\eqref{Three_D_Dirichlet-32} and to \eqref{Lower_Bound-3}. Note that in this case $b_j=\pi\rho j/\delta$ and~\eqref{Expli-Diri-Solu-3-D} follows. This completes the proof of statement (C).
\end{proof}

The proof of Theorem~\ref{Lower-and-Upper-bound-Theorem} is now complete as well.
\end{proof}

Now we are ready to prove theorem~\ref{Arbitrary-Domain-Estimation-Thm} from page~\pageref{Arbitrary-Domain-Estimation-Thm}.

\begin{proof}[Proof of Theorem\,\,\ref{Arbitrary-Domain-Estimation-Thm}]
        \label{Proof-of-Arbitrary-Estimation}
    By the Rayleigh's Theorem in \cite{Chavel}, p.~16, the relationship \eqref{Disc-Manifold-Inclusion-Relationship}, p.~\pageref{Disc-Manifold-Inclusion-Relationship} yields
    \begin{equation}\label{Eigenvalue-Disc-Manifold-Inequalities}
        \lambda_{\min}(D^n_2)\leq\lambda_{\min}(M^n)\leq\lambda_{\min}(D^n_1)\,.
    \end{equation}
    For $n>3$, using \eqref{Lower_Bound-4} from p.~\pageref{Lower_Bound-4} together with \eqref{Eigenvalue-Disc-Manifold-Inequalities}, we have
    \begin{equation}
        -\kappa\frac{(n-1)^2}{4}+\left(\frac{2\pi}{d_2}\right)^2\leq
        \lambda_{\min}(D^n_2)\leq\lambda_{\min}(M^n)\,,
    \end{equation}
    which completes the proof of Statement~(C) of Theorem~\ref{Arbitrary-Domain-Estimation-Thm}, p.~\pageref{Arbitrary-Domain-Estimation-Thm}.

    The similar argument together with formulae~\eqref{Lower_Bound-2}, \eqref{Lower_Bound-3}, p.~\pageref{Lower_Bound-2} and~\eqref{Eigenvalue-Disc-Manifold-Inequalities} yields directly Statements (A) and (B) of Theorem~\ref{Arbitrary-Domain-Estimation-Thm}, p.~\pageref{Arbitrary-Domain-Estimation-Thm}. This completes the proof of Theorem~\ref{Arbitrary-Domain-Estimation-Thm}.
\end{proof}


\bibliographystyle{alpha}

\end{spacing}

\def\changemargin#1#2{\list{}{\rightmargin#2\leftmargin#1}\item[]}
\let\endchangemargin=\endlist

\begin{changemargin}{10.5cm}{0cm}

Department of Mathematics \\Central Connecticut State University \\New Britain, Connecticut 06050 \\E-mail address: s.artamoshin@ccsu.edu

\end{changemargin}


\end{document}